\documentclass[twoside,10pt, letterpaper,reqno]{amsart}
\usepackage[margin=1.0in]{geometry}
\usepackage[utf8]{inputenc}
\usepackage{setspace}
\usepackage{amsaddr}
\usepackage{amsmath}
\usepackage{amssymb}
\usepackage{amsthm}
\usepackage{amsfonts}
\usepackage{float}
\usepackage{mathtools}
\usepackage{enumerate}
\usepackage{caption}
\usepackage[section]{placeins}
\usepackage{etoolbox} 
\usepackage{tikz}
\usepackage{color}

\patchcmd{\subequations}{{0}}{{-1}}{}{}       
\patchcmd{\subequations}{\alph}{.\arabic}{}{} 

\usepackage{breqn}

\newtheorem{theorem}{Theorem}[section]

\newtheorem{lemma}[theorem]{Lemma}

\title{Virus Dynamics on $k$-Level Starlike Graphs}

\author{Akihiro Takigawa}
\address{Department of Applied Mathematics, University of Waterloo}
\email{atakigaw@uwaterloo.ca}

\author{Steven J. Miller}
\address{Department of Mathematics and Statistics, Williams College}
\email{sjm1@williams.edu}

\begin{document}

\maketitle

Becker, Greaves-Tunnell, Kontorovich, Miller, Ravikumar, and Shen determined the long term evolution of virus propagation behavior on a hub-and-spoke graph of one central node and $n$ neighbors, with edges only from the neighbors to the hub (a $2$-level starlike graph), under a variant of the discrete-time SIS (Suspectible Infected Suspectible) model. The behavior of this model is governed by the interactions between the infection and cure probabilities, along with the number $n$ of $2$-level nodes. They proved that for any $n$, there is a critical threshold relating these rates, below which the virus dies out, and above which the probabilistic dynamical system converges to a non-trivial steady state (the probability of infection for each category of node stabilizes). For $a$, the probability at any time step that an infected node is not cured, and $b$, the probability at any time step that an infected node infects its neighbors, the threshold for the virus to die out is $b \leq (1-a)/\sqrt{n}$. We extend this analysis to $k$-level starlike graphs for $k \geq 3$ (each $(k-1)$-level node has exactly $n_k$ neighbors, and the only edges added are from the $k$-level nodes) for infection rates above and below the critical threshold of $(1-a)/\sqrt{n_1+n_2+\dots+n_{k-1}}$. We do this by first analyzing the dynamics of nodes on each level of a $3$-level starlike graph, then show that the dynamics of the nodes of a $k$-level starlike graph are similar, enabling us to reduce our analysis to just $3$ levels, using the same methodology as the $3$-level case.

\tableofcontents

\hfill\\
\section{Introduction}
The problem of studying the node-states within an interconnected network of nodes is of great interest across many disciplines, with applications such as the modeling of distributed systems in computer science, the study of societal dynamics in sociology, as well as the study of epidemic and pandemic scenarios in epidemiology. For both infectious diseases and computer viruses, while the mechanism of transmission may be described in simple terms, how this interacts in a population or a computer network is a very complex problem that is difficult to comprehend. A model of virus propagation (often called an epidemiological model) uses microscopic descriptions of a virus, such as the role of an infectious node, to predict the large-scale behavior of the virus' spread throughout a network. In many scientific fields it is often possible to conduct experiments to test hypotheses and obtain data. However, experiments with infectious diseases in human populations, as well as computer viruses on public computer networks are dangerous and unethical, making this impossible in many situations. Mathematical models can be used instead to perform theoretical experiments. For example, epidemiological models have been used to compare the spread of different diseases such as measles, whooping cough and polio in London and Yorke \cite{ly}, Yorke et. al \cite{ynpm} and Anderson and May \cite{am}. Epidemiological models have also proven invaluable in the fight against the COVID-19 pandemic starting in December 2019, in studies such as \cite{cma}, \cite{aerra}, \cite{clcl}.  Epidemiological models can also be useful in answering difficult problems where both scientific and socioeconomic factors must be taken into account, such as finding optimal vaccination strategies while minimizing deaths and monetary cost in Longini et. al \cite{loacel}.

One of the most widely used classes of epidemiological models is derived from the SIR (Susceptible Infected Removed) model introduced by W. O. Kermack and A. G. McKendrick \cite{km} in 1927. This model assumes that any node in a given network are in one of three states-- susceptible (S), when a node is not infected but can be infected, infected (I), when a node is a carrier of the disease and can infect susceptible neighboring nodes, or removed (R), when a node can no longer be infected, either through immunity to the disease or another mechanism of removal from the network, such as death in the case that the network mimics a population. For infectious diseases, the basic SIR model was then further developed into two extensions called the SIR model with vital dynamics, where population dynamics such as the natural birth and death rate are taken into account, commonly used for diseases with slower spread, and the SIR model without vital dynamics where birth and death rates are not considered, commonly used for diseases that spread so fast that the natural birth and death rates of a population have a negligible effect on the trajectory of an epidemic. Finally, a third variant called the SIS (Susceptible Infected Susceptible) model which does not have the removed (R) state is commonly used to model diseases that do not confer long-lasting immunity, such as the common cold. This is the kind of model that we focus on in this paper. Over the years, numerous other extensions to the SIR model have been introduced. Its simple structure and its wide range of applications have enabled the SIR model to remain a popular choice almost a century after its introduction.

Y. Wang, C. Deepayan, C. Wang and C. Faloutsos \cite{wdwf} proposed a discrete-time SIS model which depends on local node interactions as a model which strongly mimics network topologies found in the real world. Many epidemiological models for infectious diseases assume homogeneous connectivity, where every individual has equal contact with everyone else in the population. This makes random graphs such as ones generated by Erdős–Rényi processes a suitable topology for many real world populations. However, many computer networks are reported to be scale-free, hence follow a power-law structure instead. The merit of this proposed model is that it makes no assumption of homogeneity nor the network topology, enabling it to be applicable in a wide range of scenarios. In this model, each node is either Susceptible (S) or Infected (I) at any time-step. A susceptible node is healthy, but at any time-step can be infected by its neighbors. On the other hand, at any time-step an infected node can also be cured and revert to the susceptible state. The model parameters are $\beta$, the probability at any time-step that an infected node infects its neighbors, and $\delta$, the probability at any time-step that an infected node is cured.\\

In studying this model, the following are the key questions posed. \
\begin{enumerate} [(1)]
\item
Given a set of model parameters and a particular initial state, does the system eventually reach a steady state?
\item
If the system does reach a steady state, what are the characteristics of that state?
\item
What is the rate of convergence of the system?
\end{enumerate}

For the model proposed by \cite{wdwf}, they gave a heuristic argument for a sufficient criterion for the node infection probabilities to converge to a trivial solution, so that the infection dies out. For star graphs with one central hub node connected to $n$ spoke nodes, this condition is $b \leq (1-a)/\sqrt{n}$, where $a \ = \ 1-\delta$ and $b \ = \ \beta$. Becker et al. \cite{bgkmrs} then gave a theoretical foundation to this argument by showing that the SIS model exhibits phase transition behavior, and that this threshold is both necessary and sufficient. Hence, below this threshold the virus dies out, and above the system converges to a non-trivial steady state independent of the steady state, given that the initial state is non-trivial.

\cite{bgkmrs} conjectured that similar behavior holds on starlike graphs with more ``levels''. Each additional level of a starlike graph is an additional layer of spoke nodes connected to each of the previous level of spoke nodes. For example, the star graph as described above is considered a $2$-level starlike graph because it has one central hub node (level $1$) connected to $n$ spoke nodes (level $2$). A $3$-level starlike graph has one central hub node (level $1$) connected to $n_1$ spoke nodes (level $2$), and each of the $n_1$ spoke nodes are connected to $n_2$ additional spoke nodes (level $3$).

\begin{figure}[hb]
\includegraphics[width=0.5\linewidth]{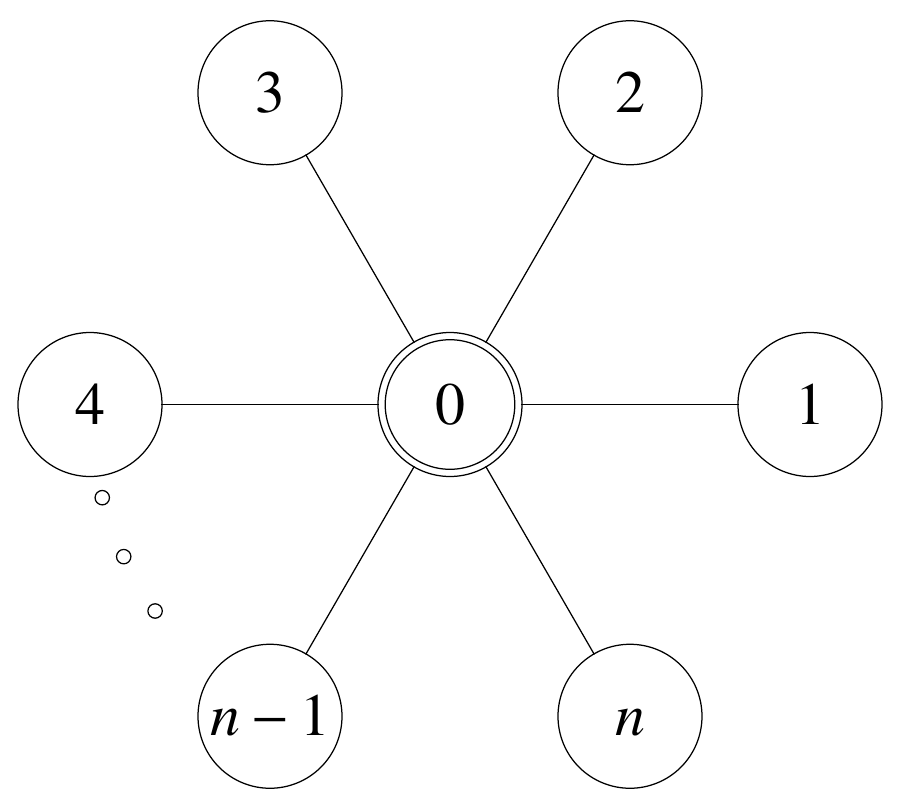}
\caption{A $2$-level starlike graph.}
\end{figure}

Starlike graphs mimic many real-world network topologies. For example, it can represent a metropolitan area with one central population hub and a number of dependent areas when modeling the spread of infectious diseases. In computer science, starlike networks are common in distributed systems, allowing us to model the spread of computer viruses. Hence the application of the proposed model to starlike graphs is particularly interesting to analyze.

We provide a theoretical basis for the existence of a critical threshold for virus propagation on starlike graphs. We first show this on the $3$-level starlike graph, and then proceed to naturally extend the same ideas to starlike graphs with an arbitrary number of levels-- a $k$-level starlike graph. Namely, we show for a $k$-level starlike graph, a critical threshold of $b \leq (1-a)/\sqrt{n_1+\dots+n_{k-1}}$ exists, below which the virus dies out. \\

\begin{subequations}
\section{Preliminaries}
\subsection{The Proposed Model}
\hfill\\

\cite{wdwf} proposed the following propagation model. Denote by $\beta$, the probability at any time-step that an infected node infects its neighbors, and by $\delta$, the probability at any time-step that an infected node is cured.

If $p_{i,t}$ is the probability a node $i$ is infected at time $t$, the SIS model is governed by the following equation:
\begin{equation}
1-p_{i,t}  \ = \ (1-p_{i,t-1})\zeta_{i,t}+\delta p_{i,t}\zeta_{i,t}, \label{model_equation}
\end{equation}
where $\zeta_{i,t}$ is the probability that a node $i$ is not infected by its neighbors at time $t$. We can write $\zeta_{i,t}$ as follows:
\begin{equation}
\zeta_{i,t} \ = \ \prod_{j \sim i} p_{j,t-1}(1-\beta)+(1-p_{j,t-1}) \ = \ \prod_{j \sim i} (1-\beta p_{j,t-1}), \label{zeta_equation}
\end{equation}
where $j \sim i$ means $i$ and $j$ are neighbors. A node $i$ is assumed to be healthy at time $t$ if it satisfies one of the following conditions:
\begin{itemize}
\item
$i$ was healthy at time $t-1$ or prior, and is not infected by its neighbors at time $t$ as governed by $\zeta_{i,t}$, or
\item
$i$ was infected at time $t-1$ or prior, but was cured at time $t$, and is not infected by its neighbors at time $t$ as governed by $\zeta_{i,t}$.
\end{itemize}

\subsection{Basic Result on the Infection Rates for Nodes on Common Levels}
\hfill\\

We start by stating a crucial observation that makes a theoretical investigation of our dynamical system feasible. We consider graphs with a starlike graph topology, as described in the introduction. Suppose the graph has $n_1$ $2$-level nodes, $n_2$ $3$-level nodes, $\dots$ , $n_{k-1}$ $k$-level nodes, along with one central hub node. The central hub node is numbered $0$, for a total of $n_1+\dots+n_{k-1}+1$ nodes.
\begin{lemma} \label{common_behavior_lemma}
For any initial configuration, as time evolves all nodes on the same level converge to a common behavior.
\end{lemma}
\begin{proof}
Let $a = 1-\delta$ and $b = \beta$. From (\ref{model_equation}) and (\ref{zeta_equation}), we get
\begin{align}
p_{i,t} & \ = \ (1-p_{i,t-1})\prod_{j \sim i}(1-\beta p_{j,t-1})+\delta p_{i,t}\prod_{j \sim i}(1-\beta p_{j,t-1}) \nonumber \\
& \ = \ 1-(1-ap_{i,t-1})\prod_{j \sim i}(1-bp_{j,t-1}).
\end{align}
Then, the nodes at each level become the following.
\paragraph{Start Node (level 1)}
We have
\begin{equation}
p_{0,t}  \ = \  1-(1-ap_{0,t-1})\prod_{j \ = \ 1}^{n_1}(1-bp_{j,t-1}).
\end{equation}
The claim is trivial for the start node, as there is only one start node.
\paragraph{Middle Nodes (levels $2$ through $k-1$)}
For a $m$-level spoke node $i$,
\begin{equation}
p_{i,t}  \ = \  1-(1-ap_{i,t-1})\prod_{j \ = \ 1}^{n_{m+1}+1}(1-bp_{j,t-1}),
\end{equation}
where the nodes $j$ are the $n_{m+1}$ $m+1$-level neighbors of node $i$, in addition to the single $m-1$-level node adjacent to node $i$.
We prove that all the spokes at level $m$ assume identical values by showing that for two nodes $i$ and $k$ on level $m$, $|p_{i,t} - p_{k,t}| \to 0$ as $t \to \infty$. We have
\begin{align}
p_{i,t} - p_{k,t} & \ = \  (p_{i,t-1}-p_{k,t-1}) a\prod_{j \ = \ 1}^{n_{m+1}+1}(1-bp_{j,t-1}) \nonumber \\
& \ = \  (p_{i,t-1}-p_{k,t-1}) (1-\delta)\prod_{j \ = \ 1}^{n_{m+1}+1}(1-\beta p_{j,t-1}).
\end{align}
Thus
\begin{equation}
|p_{i,t} - p_{k,t}|  \ = \  (p_{i,0}-p_{k,0}) (1-\delta)^t\prod_{j \ = \ 1}^{n_{m+1}+1}(1- \beta p_{j,t-1})^t.
\end{equation}
As we assume a non-trivial initial configuration, $\beta, \delta \in (0,1)$. Then, we get
\begin{equation}
\lim_{t \to \infty} |p_{i,t} - p_{k,t}| \ = \ \lim_{t \to \infty} (p_{i,0}-p_{k,0}) (1-\delta)^t\prod_{j \ = \ 1}^{n_{m+1}+1}(1- \beta p_{j,t-1})^t \ = \ 0.
\end{equation}
\paragraph{Final Nodes (level $k$)}
For a $k$-level spoke $i$,
\begin{equation}
p_{i,t}  \ = \  1-(1-ap_{i,t-1})(1-bp_{j,t-1})
\end{equation}
(where $j$ is the spoke node at level $k-1$ that $i$ is connected to). We prove that all the spokes at level $k$ assume identical values by showing that for two nodes $i$ and $k$ on level $m$, $|p_{i,t} - p_{k,t}| \to 0$ as $t \to \infty$. We have
\begin{align}
p_{i,t} - p_{k,t} & \ = \  (p_{i,t-1}-p_{k,t-1})  \cdot a(1-bp_{j,t-1}) \nonumber \\
& \ = \  (p_{i,t-1}-p_{k,t-1}) (1-\delta)(1-\beta p_{j,t-1}).
\end{align}
Thus
\begin{equation}
|p_{i,t} - p_{k,t}|  \ = \  (p_{i,0}-p_{k,0}) (1-\delta)^t(1- \beta p_{j,t-1})^t.
\end{equation}
As we assume a non-trivial initial configuration, $\beta, \delta \in (0,1)$. Then, we get
\begin{equation}
\lim_{t \to \infty} |p_{i,t} - p_{k,t}| \ = \ \lim_{t \to \infty} (p_{i,0}-p_{k,0}) (1-\delta)^t (1- \beta p_{j,t-1})^t \ = \ 0.
\end{equation}
\end{proof}
This observation allows us to simplify our model to a model in terms of the infection probabilities of a node on each level. We begin with a $3$-level starlike graph model, which we later generalize to the $k$-level case. Let $x_t$ be the probability that the hub is infected at time $t$, $y_t$ be the probability that a $2$-level spoke node is infected at time $t$, and $z_t$ be the probability that a $3$-level spoke node is infected at time $t$. These then evolve according to
\begin{equation}
\begin{pmatrix}
x_{t+1} \\
y_{t+1} \\
z_{t+1}
\end{pmatrix}  \ = \  F
\begin{pmatrix}
x_t \\
y_t \\
z_t
\end{pmatrix}
\end{equation}
where
\begin{equation}
F\begin{pmatrix}
x \\
y \\
z
\end{pmatrix} \ = \ \begin{pmatrix}
f_1(x,y,z) \\
f_2(x,y,z)\\
f_3(x,y,z)
\end{pmatrix} \ = \ \begin{pmatrix}
1-(1-ax)(1-by)^{n_1}\\
1-(1-ay)(1-bx)(1-bz)^{n_2}\\
1-(1-az)(1-by)
\end{pmatrix}.
\end{equation}
This map simplifies our model greatly, and is the foundation of our investigation. In the next section, we present our main result for $3$-level starlike graphs, which provides a theoretical basis for the critical threshold proposed by \cite{wdwf} on $3$-level starlike graphs. We then present an extension of this to starlike graphs with more than $3$ levels.

\hfill\\
\section{A Theoretical Foundation on $3$-level Starlike Graphs}

\subsection{The Method of Fixed Points}
\hfill\\

In exploring the question of whether a dynamical system reaches a steady-state, finding the fixed points of the system is important. Given a map $F: [0,1] \to [0,1]$, a fixed point is a point $(x,y,z)$ such that $F(x,y,z)  \ = \  (x,y,z)$. If it can be shown that successive iterations of a dynamical system approaches a fixed point, then we say that the dynamical system reaches a steady-state. For $3$-level starlike graphs, these two problems were thoroughly investigated. and the following is one of our main results.

\begin{theorem} \label{3-level_theorem}
\begin{enumerate}[I.]
\item
If $b \leq (1-a)/\sqrt{n_1+n_2}$, then
\begin{enumerate}[(a)]
\item
the unique fixed point of $F$ is $(0,0,0)$, and
\item
the system converges to this fixed point, in other words, the virus dies out.
\end{enumerate}
\item
If $b > (1-a)/\sqrt{n_1+n_2}$, then
\begin{enumerate}[(a)]
\item
$F$ has a unique, non-trivial fixed point $(x_f,y_f,z_f)$, and
\item
the system converges to this non-trivial fixed point.
\end{enumerate}
\end{enumerate}
\end{theorem}
The proof of this theorem is given over the next few subsections. In the next subsection, we prove parts I(a) and II(a) by determining the fixed points of $F$. Using convexity arguments, we show that the trivial fixed point is the only fixed point if $b \leq (1-a) / \sqrt{n_1+n_2}$, but there is a unique, additional fixed point for larger $b$. In order to illustrate the method used to show convergence to the fixed points, we then give a proof of I(b), namely that for $b \leq (1-a)/\sqrt{n_1+n_2}$, all initial configurations evolve to the trivial fixed point.
\end{subequations}

\subsection{Determination of Fixed Points of $F$}
\hfill\\

\begin{subequations}
We determine the behavior of fixed points of $F$ as a function of $a$, $b$, $n_1$ and $n_2$, thus proving Theorem \ref{3-level_theorem} I(a) and II(a). To do this, we look for partial fixed points, namely points where one of the $x$, $y$, or $z$-coordinates are invariant under $F$. These points are defined by the functions
\begin{align}
\phi_1(y,z) &\ = \ x \ = \ \frac{1-(1-by)^{n_1}}{1-a(1-by)^{n_1}}, \\
\phi_2(x,z) &\ = \ y \ = \ \frac{1-(1-bx)(1-bx)^{n_2}}{1-a(1-bx)(1-bx)^{n_2}}, \\
\phi_3(z,y) &\ = \ z \ = \ \frac{by}{1-a+aby},
\end{align}
whose intersections are the fixed points of $F$.

\begin{figure}[!htb]
\minipage{0.32\textwidth}
  \includegraphics[width=\linewidth]{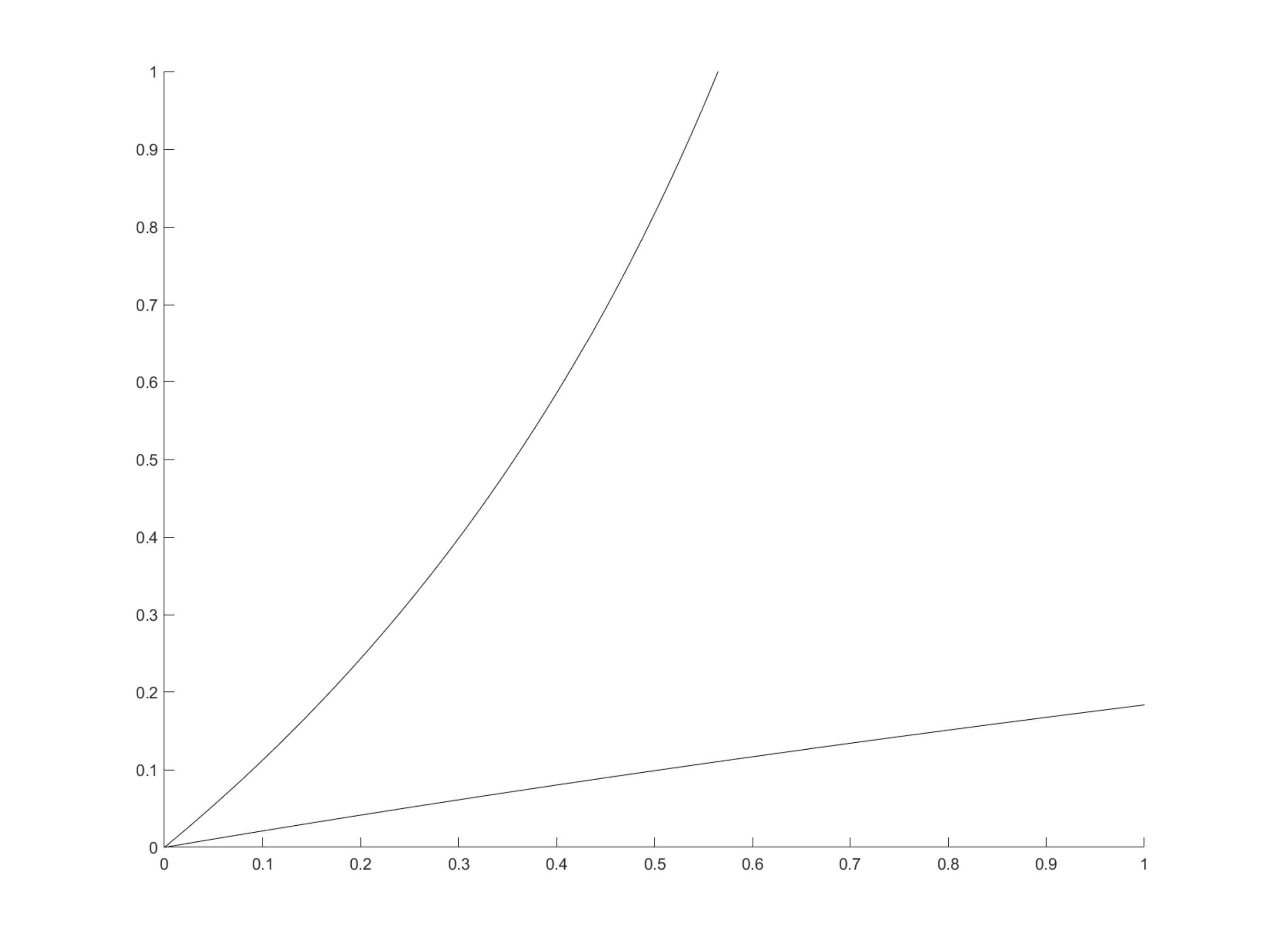}
  \begin{center}
  $b < (1-a)/\sqrt{n_1+n_2}$ \label{fig:trivial}
  \end{center}
\endminipage\hfill
\minipage{0.32\textwidth}
  \includegraphics[width=\linewidth]{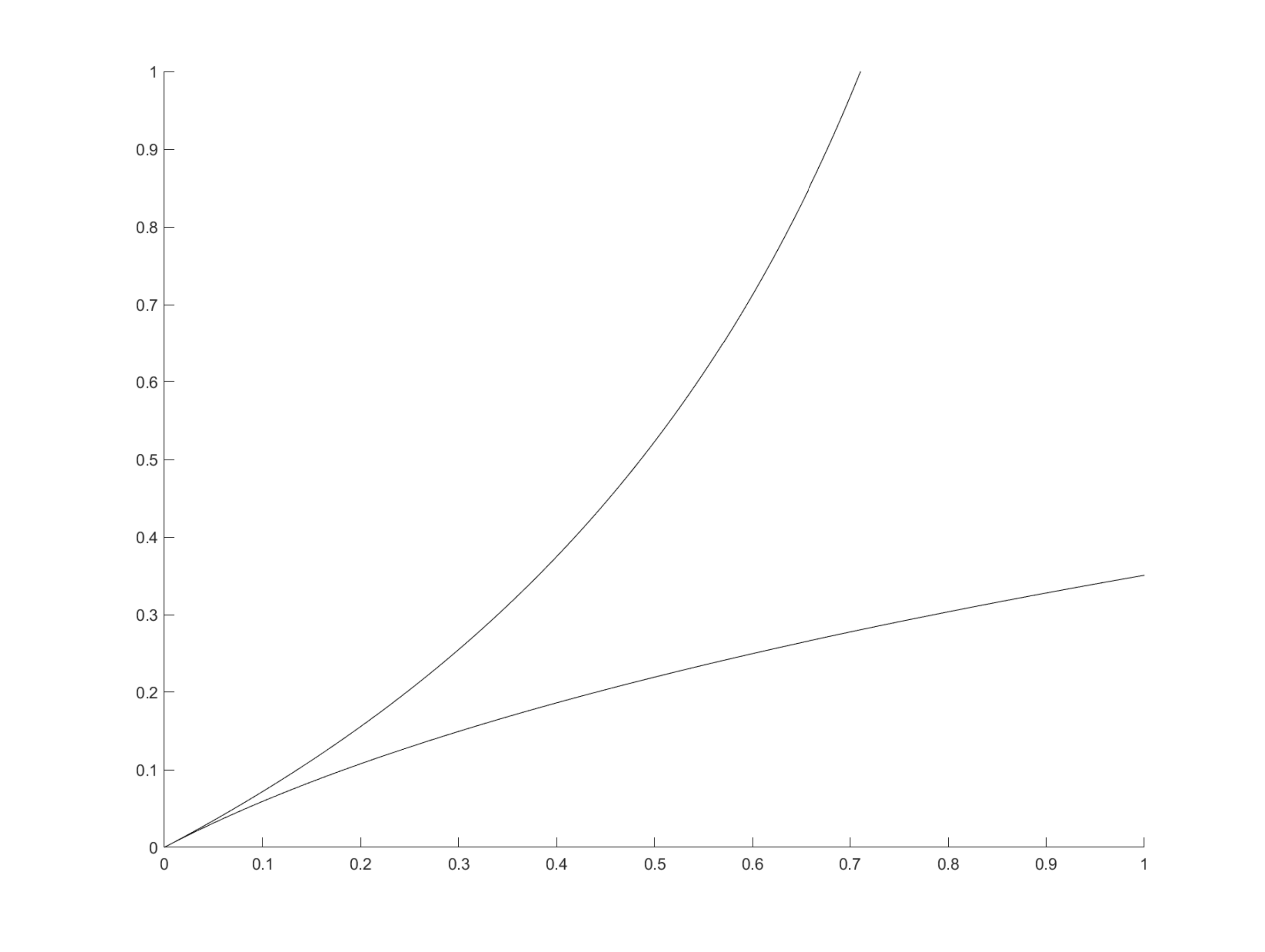}
  \begin{center}
  $b  \ = \  (1-a)/\sqrt{n_1+n_2}$ \label{fig:exact}
  \end{center}
\endminipage\hfill
\minipage{0.32\textwidth}%
  \includegraphics[width=\linewidth]{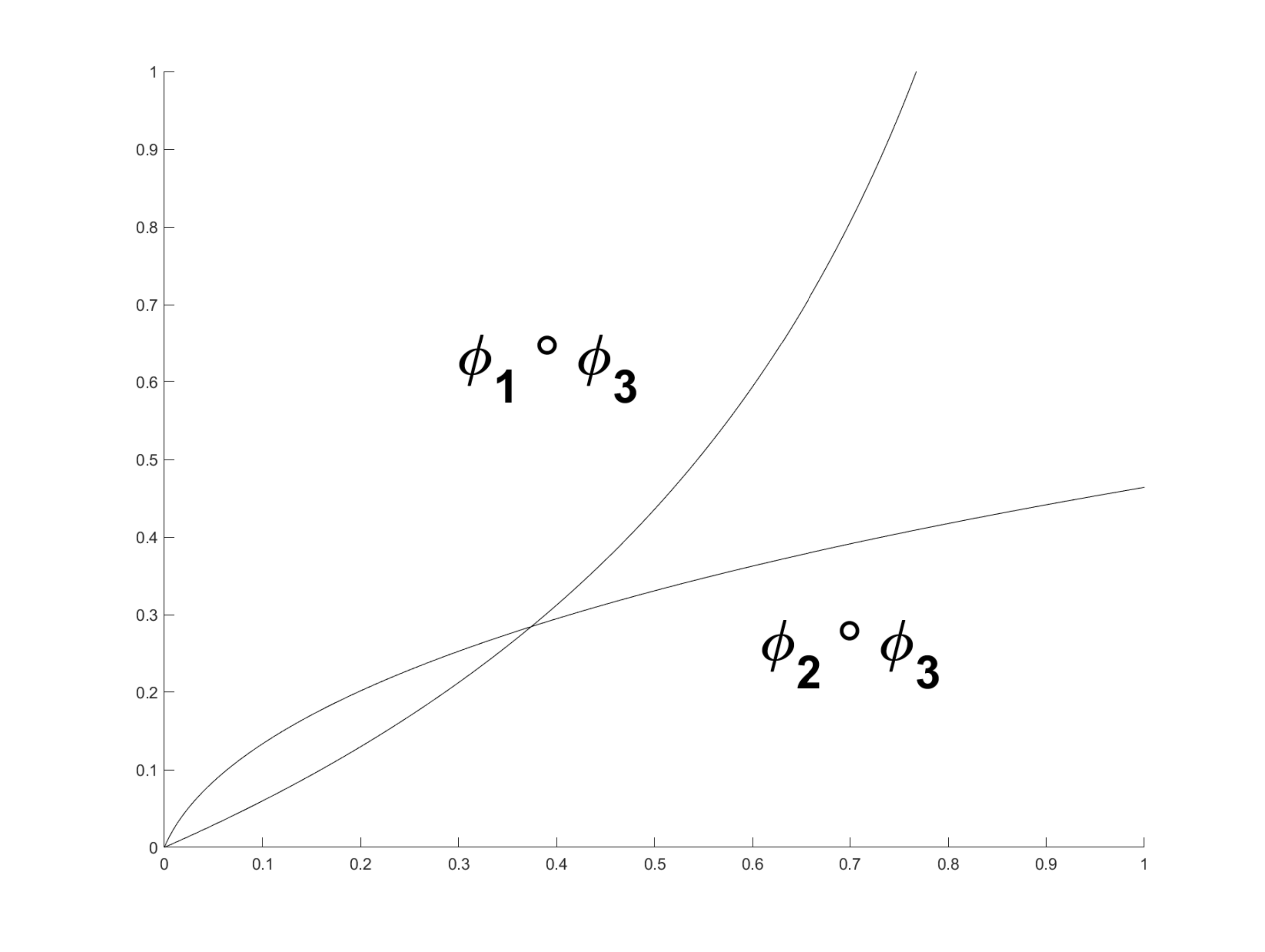}
  \begin{center}
  $b > (1-a)/\sqrt{n_1+n_2}$ \label{fig:nontrivial}
  \end{center}
\endminipage
\caption{\small Partial fixed points from $\phi_1 \circ \phi_3$ and $\phi_2 \circ \phi_3$ (From left to right, $b = 0.08,0.125,0.15, n_1 = 6, n_2 = 10, a = 0.5$).}
\label{fig:intersection_curves}
\normalsize
\end{figure}

\FloatBarrier

The convexity properties of $\phi_1$, $\phi_2$ and $\phi_3$ were analyzed by Becker et. al. \cite{bgkmrs}. It is known that $\phi_1(y,z)$ is convex, while $\phi_2(x,z)$ and $\phi_3(z,y)$ are concave. We extend their result by additionally showing that $\phi_2$ is non-decreasing in each argument, obtaining the following lemma:
\begin{lemma} \label{3_level_fixed_points_lemma}
Consider the map $F$.
\begin{enumerate}[(1)]
\item
There exists a continuous, twice differentiable convex function $\phi_1:[0,1] \to [0,1]$ such that for each $y \in [0,1]$ and $z \in [0,1]$, there is a $y' \in [0,1]$ and $z' \in [0,1]$ with $F(\phi_1(y,z),y,z)  \ = \  F(\phi_1(y,z),y',z')$.
\item
There exists a concave function $\phi_2:[0,1]^2 \to [0,1]$ that is non-decreasing in each variable such that for each $x \in [0,1]$ and $z \in [0,1]$, there is a $x' \in [0,1]$ and $z' \in [0,1]$ with $F(x,\phi_2(x,z), z)  \ = \  F(x',\phi_2(x,z), z')$.
\item
There exists a continuous, twice differentiable concave function $\phi_3:[0,1] \to [0,1]$ such that for each $x \in [0,1]$ and $y \in [0,1]$, there is a $x' \in [0,1]$ and $y' \in [0,1]$ with $F(x,y,\phi_3(x,y))  \ = \  F(x',y',\phi_3(x,y))$.
\end{enumerate}
\end{lemma}
Note that we do not require $\phi_2$ to be continuous, nor to be twice differentiable. This is because we are restricted to the domain $[0,1]^2$, while $\phi_2$ extended to $\mathbb{R}^2$ often takes values outside of $[0,1]$ with $x,z \in [0,1]$, resulting in discontinuity when observed in the domain $[0,1]^2$. The important detail is that this function is concave, and non-decreasing in each argument.
\begin{proof}
$(1)$ and $(3)$ follow from [BGKMRS]. We get
\begin{align}
&\phi_1(y,z)  \ = \  x  \ = \   \frac{1-(1-by)^{n_1}}{1-a(1-by)^{n_1}} \label{phi_1} \\ \nonumber
&\phi_3(x,y)  \ = \  z  \ = \  \frac{by}{1-a+aby} \label{phi_3}.
\end{align}
We proceed to show $(2)$. We define
\begin{equation}
g_2(x,y,z)  \ = \  1-(1-ay)(1-bx)(1-bz)^{n_2}-y \label{g_2}.
\end{equation}
We first analyze the set of $3$-tuples $(x,y,z) \in [0,1]^3$ where $g_2(x,y,z) \ = \ 0$. We see that $g_2(0,0,0)  \ = \  0$, $g_2(x,0,z) > 0$ for $x \in (0,1]$, $z \in (0,1]$, and $g_2(x,1,z) < 0$ for $x \in [0,1]$, $z \in [0,1]$. Hence, by the Intermediate Value Theorem, for each $x \in (0,1]$ and $z \in (0,1]$ there is a number (which we denote by $\phi_2(x,z)$) such that $g_2(x,\phi_2(x,z),z) \ = \ 0$, $\phi_2(x,z) \in [0,1]$, and $\phi_2(x,z)$ is a continuous and differentiable function of $x$ and $z$:
\begin{equation}
\phi_2(x,z)  \ = \  y  \ = \  \frac{1-(1-bx)(1-bz)^{n_2}}{1-a(1-bx)(1-bz)^{n_2}}. \label{phi_2}
\end{equation}
Next, we show that $\phi_2(x,z)$ is non-decreasing in both $x$ and $z$. Let $\alpha \in [0,1]$, $\beta \in [0,1]$ with $\alpha > \beta$. Then
\begin{multline}
\frac{1-(1-b\alpha)(1-bz)^{n_2}}{1-a(1-b\alpha)(1-bz)^{n_2}}-\frac{1-(1-b\beta)(1-bz)^{n_2}}{1-a(1-b\beta)(1-bz)^{n_2}} \ = \ \\
\frac{\splitfrac{(1-a(1-b\beta)(1-bz)^{n_2})(1-(1-b\alpha)(1-bz)^{n_2})}{-(1-a(1-b\alpha)(1-bz)^{n_2})(1-(1-b\beta)(1-bz)^{n_2})}}{(1-a(1-b\alpha)(1-bz)^{n_2})(1-a(1-b\beta)(1-bz)^{n_2})}. \label{non-decreasing_calc_start}
\end{multline}
Expanding the numerator, we find
\begin{multline}
(1-a(1-b\beta)(1-bz)^{n_2})(1-(1-b\alpha)(1-bz)^{n_2})-(1-a(1-b\alpha)(1-bz)^{n_2})(1-(1-b\beta)(1-bz)^{n_2}) \\
 \ = \  (1-(1-b\alpha)(1-bz)^{n_2}-a(1-b\beta)+a(1-b\alpha)(1-b\beta)(1-bz)^{2n_2}) \\
-(1-(1-b\beta)(1-bz)^{n_2}-a(1-b\alpha)(1-bz)^{n_2}+a(1-b\alpha)(1-b\beta)(1-bd)^{2n_2}) > 0,
\end{multline}
while expanding the denominator gives
\begin{multline}
(1-a(1-b\alpha)(1-bz)^{n_2})(1-a(1-b\beta)(1-bz)^{n_2}) \\
 \ = \  1-a(1-b\beta)(1-bz)^{n_2}-a(1-b\alpha)(1-bz)^{n_2}+a^2(1-b\alpha)(1-b\beta)(1-bz)^{2n_2} > 0.
\end{multline}
Hence
\begin{equation}
\frac{1-(1-b\alpha)(1-bz)^{n_2}}{1-a(1-b\alpha)(1-bz)^{n_2}}-\frac{1-(1-b\beta)(1-bz)^{n_2}}{1-a(1-b\beta)(1-bz)^{n_2}} > 0,
\end{equation}
so $\phi_2(x,z)$ is non-decreasing in $x$.\newline
Similarly,
\begin{multline}
\frac{1-(1-bx)(1-b\alpha)^{n_2}}{1-a(1-bx)(1-b\alpha)^{n_2}}-\frac{1-(1-bx)(1-b\beta)^{n_2}}{1-a(1-bx)(1-b\beta)^{n_2}} \ = \ \\
\frac{(1-(1-bx)(1-b\alpha)^{n_2})(1-a(1-bx)(1-b\beta)^{n_2})-(1-a(1-bx)(1-b\alpha)^{n_2})(1-(bx)(1-b\beta)^{n_2})}{(1-a(1-bx)(1-b\alpha)^{n_2})(1-a(1-bx)(1-b\beta)^{n_2})}.
\end{multline}
Expanding the numerator yields
\begin{multline}
(1-(1-bx)(1-b\alpha)^{n_2})(1-a(1-bx)(1-b\beta)^{n_2})-(1-a(1-bx)(1-b\alpha)^{n_2})(1-(1-bx)(1-b\beta)^{n_2}) \\
 \ = \  (1-(1-bx)(1-b\alpha)^{n_2}-a(1-bx)(1-b\beta)^{n_2}+a(1-bx)^2(1-b\alpha)^{n_2}(1-b\beta)^{n_2})\\
-(1-(1-bx)(1-b\beta)^{n_2}-a(1-bx)(1-b\alpha)^{n_2}+a(1-bx)^2(1-b\alpha)^{n_2}(1-b\beta)^{n_2}) > 0,
\end{multline}
while expanding the denominator gives
\begin{multline}
(1-a(1-bx)(1-b\alpha)^{n_2})(1-a(1-bx)(1-b\beta)^{n_2}) \\
 \ = \  1-a(1-bx)(1-b\beta)^{n_2}-a(1-bx)(1-b\alpha)^{n_2}+a^2(1-b\alpha)(1-b\beta)(1-bx)^{2n_2} > 0,
\end{multline}
Hence
\begin{equation}
\frac{1-(1-bx)(1-b\alpha)^{n_2}}{1-a(1-bx)(1-b\alpha)^{n_2}}-\frac{1-(1-bx)(1-b\beta)^{n_2}}{1-a(1-bx)(1-b\beta)^{n_2}} > 0, \label{non-decreasing_calc_end}
\end{equation}
so $\phi_2(x,z)$ is also non-decreasing in $z$. \newline
We refer to [BGKMRS] for the concavity of $\phi_2$.
\end{proof}

We now have three surfaces $x = \phi_1(y,z)$, $y = \phi_2(x,z)$, and $z = \phi_3(x,y)$ which respectively represent the sets of partial fixed points where $x$, $y$, and $z$ are unchanged on iteration by $F$. To find the fixed points, we look at the intersection of these three surfaces. To simplify our problem further from one in three dimensions to one in two dimensions, we examine the curves formed by the intersection of $\phi_1$ and $\phi_3$, and similarly $\phi_2$ and $\phi_3$ (see: Figure \ref{fig:intersection_curves}). From here on, we refer to $\phi_2 \circ \phi_3$ as $\phi_{2,3}$ for brevity. The first curve is already done, as it is independent of $z$. For the second curve, as
\begin{equation}
x  \ = \  \frac{1}{b}+\frac{y-1}{b(1-ay)(1-bz)^{n_2}}
\end{equation}
and as $z  \ = \  by/(1-a+aby)$, we can write
\begin{equation}
\phi_{2,3}(y)  \ = \  x  \ = \  \frac{1}{b}+\frac{y-1}{b(1-ay)(1-\frac{b^2y}{1-a+aby})^{n_2}}. \label{phi_2_3}
\end{equation}
\newline\newline
The next lemma is useful for determining the concavity of $\phi_{2,3}(y)$.
\begin{lemma} \label{concavity_lemma}
Let $f:[0.1]^2 \to [0,1]$ be a concave function that is non-decreasing in each argument, and let $g_1: [0,1] \to [0,1]$, $g_2: [0,1] \to [0,1]$ be concave functions. Then, $f(g_1(x),g_2(z))$ is concave.
\end{lemma}
\begin{proof}
As $g_1$, $g_2$ are concave, for any $x_1,x_2,y_1,y_2$ and $\alpha \in [0,1]$,
\begin{equation}
    g_1((1-\alpha)x_1+\alpha y_1) \ \geq \ (1-\alpha)g_1(x_1)+ \alpha g_1(y_1)
\end{equation}
and
\begin{equation}
    g_2((1-\alpha)x_2+\alpha y_2) \ \geq \ (1-\alpha)g_2(x_2)+ \alpha g_1(y_2).
\end{equation}
As $f$ is non-decreasing in each argument,
\begin{multline}
    f(g_1((1-\alpha)x_1+\alpha y_1), g_2((1-\alpha)x_2+\alpha y_2)) \\
    \ \geq \ f((1-\alpha)g_1(x_1)+ \alpha g_1(y_1), (1-\alpha)g_2(x_2)+ \alpha g_1(y_2)).
\end{multline}
As $f$ is concave,
\begin{multline}
    f((1-\alpha)g_1(x_1)+ \alpha g_1(y_1), (1-\alpha)g_2(x_2)+ \alpha g_1(y_2)) \\
    \ \geq \ (1-\alpha)f(g_1(x_1),g_2(x_2))+\alpha f(g_1(y_1),g_2(y_2)).
\end{multline}
Hence, $f(g_1(x),g_2(z))$ is concave.
\end{proof}
We also recall Lemma 2.2 from \cite{bgkmrs}, which is useful for determining the number and location of fixed points of $F$.
\begin{lemma}[{BGKMRS}] \label{bgkmrs_lemma}
Let $h_1$, $h_2$ be twice continuously differentiable functions such that $h_1(x)$ is convex and $h_2(x)$ is concave. If there exists some $p$ such that $h_1'(p) \leq h_2'(p)$ and $h_1(p)  \ = \  h_2(p)$, then $h_1(x) \neq h_2(x)$ for all $x > p$.
\end{lemma}
We are now able to determine the location of the fixed points.
\newline\newline
\begin{proof}[{Proof of Theorem 1.1, I(a)}]
From Lemma 2.1, $\phi_1(y)$ is convex, $\phi_2(x,z)$ is concave and non-decreasing in each argument, and $\phi_3(y)$ is concave. Then, as
\begin{equation}
\phi_2(x,\phi_3(x,y)) \ = \ y \ = \ \frac{1-(1-bx)(1-\frac{b^2y}{1-a+aby})^{n_2}}{1-a(1-bx)(1-\frac{b^2y}{1-a+aby})^{n_2}}
\end{equation}
and $x$ and $\phi_3(x,y)$ are both concave and are defined in $[0,1] \to [0,1]$, from Lemma 2.2, $\phi_{2,3}$ is concave as well.
Writing this as $\phi_{2,3}(y)$ (see (\ref{phi_2_3})), note that
\begin{equation}
\phi_{1,3}'(0)  \ = \  \phi_1'(0)  \ = \  \frac{bn_1}{1-a},\quad \phi_{2,3}'(0)  \ = \  \frac{(1-a)^2-b^2n_2}{b(1-a)}.
\end{equation}
We see that $\phi_{2,3}'(0) > \phi_1'(0)$ when $b \leq (1-a)/\sqrt{n_1+n_2}$. Hence, by Lemma 2.3, when $b \leq (1-a)/\sqrt{n_1+n_2}$, there is no $y>0$ such that $\phi_1(y) > \phi_{2,3}(y)$. Thus, when $b \leq (1-a)/\sqrt{n_1+n_2}$, $(0,0)$ is the only point at which $\phi_1(y)$ and $\phi_{2,3}(y)$ agree in $[0,1]^2$. Also, as $z \ = \ by/(1-a+aby)$, if $y \ = \ 0$ then $z \ = \ 0$. Therefore, the trivial fixed point is the unique fixed point of $F$ in $[0,1]^3$.
\end{proof}
The next lemma, which is similar to Lemma 2.3 from \cite{bgkmrs}, is the key to proving the existence of a unique non-trivial fixed point when $b > (1-a)/\sqrt{n_1+n_2}$.
\begin{lemma} \label{bgkmrs_extended_lemma}
Let $h_1: [0,1] \to [0,1]$ be a twice continuously differentiable function such that $h_1(x)$ is convex, and let $h_2: [0,1] \to [0,1]$ be a function that is either:
\begin{enumerate}[1.]
\item
twice continuously differentiable, or
\item
is discontinuous only at $x \ = \ 0$,
\end{enumerate}
and $h_2(x)$ is concave. Furthermore, let $h_1(0)  \ = \  h_2(0)  \ = \  0$ and $h_1(x) \neq h_2(x)$ for $x>0$ sufficiently small. Then there exists at most one other $x>0$ for which $h_1(x)  \ = \  h_2(x)$.
\end{lemma}
\begin{proof}
The claim is trivial if there is only one point of intersection, so we assume there are at least two. Without loss of generality, let $p>0$ be the first point above $0$ where $h_1$ and $h_2$ agree. Such a point exists because both $h_1$ and $h_2$ are continuous in $(0,1]$ in any case, and $h_1(x) \neq h_2(x)$ for $x>0$ sufficiently small.
As $h_1(x)$ is convex, $h_1'(x)$ is increasing. By the Mean Value Theorem, there exists a point $c_1 \in (0,p)$ such that
\begin{equation}
h_1'(c_1)  \ = \  \frac{h_1(p)-h_1(0)}{p-0}  \ = \  \frac{h_1(p)}{p}.
\end{equation}
As $h_1'$ is increasing, $h_1'(p) > h_1'(c_1)$. Furthermore, $h_1'(x) > h_1(c_1)$ for all $x \geq p$. Similarly, as $h_2(x)$ is concave, $h_2'(x)$ is decreasing. By the Mean Value Theorem, there is a point $c_2 \in (0,p)$ such that
\begin{equation}
h_2'(c_2)  \ = \  \frac{h_2(p)-h_2(0)}{p-0}  \ = \  \frac{h_2(p)}{p}.
\end{equation}
As $h_2'$ is decreasing, $h_2'(p) < h_2'(c_2)$, and $h_2'(x) < h_2'(c_2)$ for all $x \geq p$. However, as $h_1(p)  \ = \  h_2(p)$, we get $h_1'(c_1) \ = \ h_2'(c_2)$, so $h_1'(x) > h_2'(x)$ for all $x \geq p$. Then, from Lemma 2.3, there cannot be another point of intersection other than $p$.
\end{proof}
We now complete our analysis.
\begin{proof}[{Proof of Theorem 1.1, II(a)}]
First, note that when $y  \ = \  0$, $x  \ = \  \phi_1(y)  \ = \  \phi_1(0)  \ = \  0$, and $x  \ = \  \phi_2(y)  \ = \  \phi_2(0)  \ = \  0$. Through direct inspection, it is easy to see that when $n_2 \geq 2$, $\phi_{2,3}$ is discontinuous at only $x \ = \ 0$, and twice continuously differentiable otherwise. We also know from the proof of Theorem 1.1 I(a) that $\phi_1$ is convex and $\phi_{2,3}$ is concave. Now, if $\phi_{2,3}$ has a discontinuity at only $x \ = \ 0$, then for $y>0$ sufficiently small, $\phi_1(y)$ is above $\phi_{2,3}(y)$ because $\phi_1$ is continuous and convex while $\phi_{2,3}$ is concave, so for some $y > 0$, $\phi_1(y) > 0$ while $\phi_{2,3}(y)  \ = \  0$. Otherwise, we know from the proof of Theorem 1.1 I(a) that $\phi_1(y)$ is above $\phi_{2,3}(y)$ near the origin, since $\phi_1'(0) > \phi_2'(0)$. As $y \ = \ \phi_{2,3}(x)$ is defined in $[0,1]$ for all $x \in [0,1]$ and $x  \ = \  \phi_1(y)$ is defined in $[0,1]$ for all $y \in [0,1]$, as $x \to 1$ $\phi_{2,3}(x)$ tends to a number strictly less than 1, thus the curve $y  \ = \  \phi_{2,3}(x)$ hits the line $x \ = \ 1$ at a point below $(1,1)$. Similarly, the curve $x \ = \ \phi_1(y)$ hits the line $y \ = \ 1$ to the left of $(1,1)$. Thus, at some point $\phi_{2,3}(y)$ flips to be above $\phi_1(y)$, so there must be a point at which the two intersect, which is a non-trivial fixed point. When the $y$ coordinate of this point is determined, so is $z$, because $z \ = \ by/(1-a+aby)$.

We now know that there exists at least two fixed points, the trivial one and a non-trivial one. By Lemma 2.4, there are no other fixed points, so there exists a unique non-trivial fixed point when $b > (1-a)/\sqrt{n_1+n_2}$.
\end{proof}
\end{subequations}

\hfill\\
\subsection{Dynamical Behavior: $b\leq(1-a)/\sqrt{n_1+n_2}$}
\hfill\\

\begin{subequations}

\subsubsection{Properties of Region I}
\begin{figure}
\minipage{0.32\textwidth}
  \includegraphics[width=\linewidth]{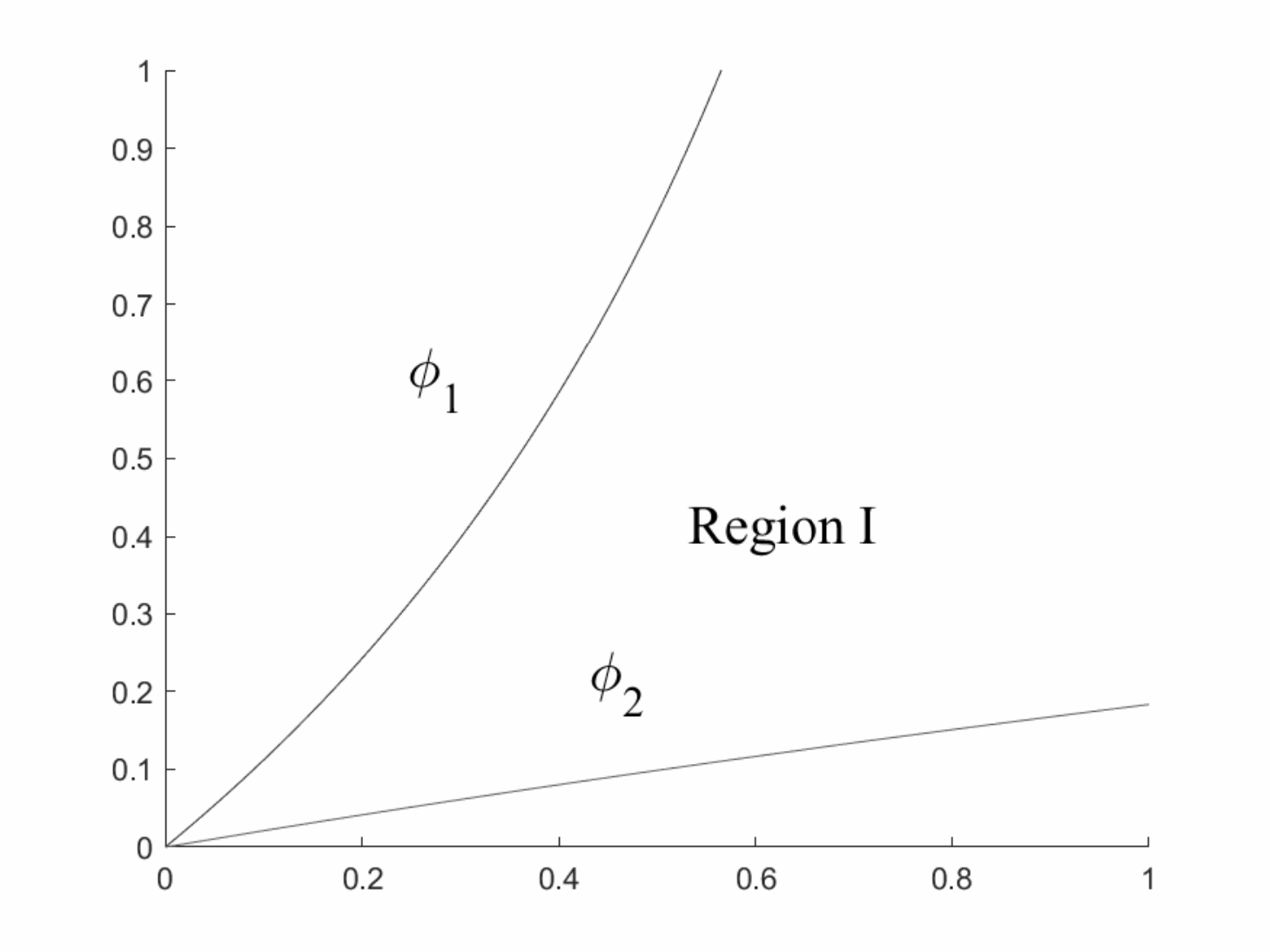}
  \begin{center}
  $z \ = \ 0$ \label{fig:trivial0}
  \end{center}
\endminipage\hfill
\minipage{0.32\textwidth}
  \includegraphics[width=\linewidth]{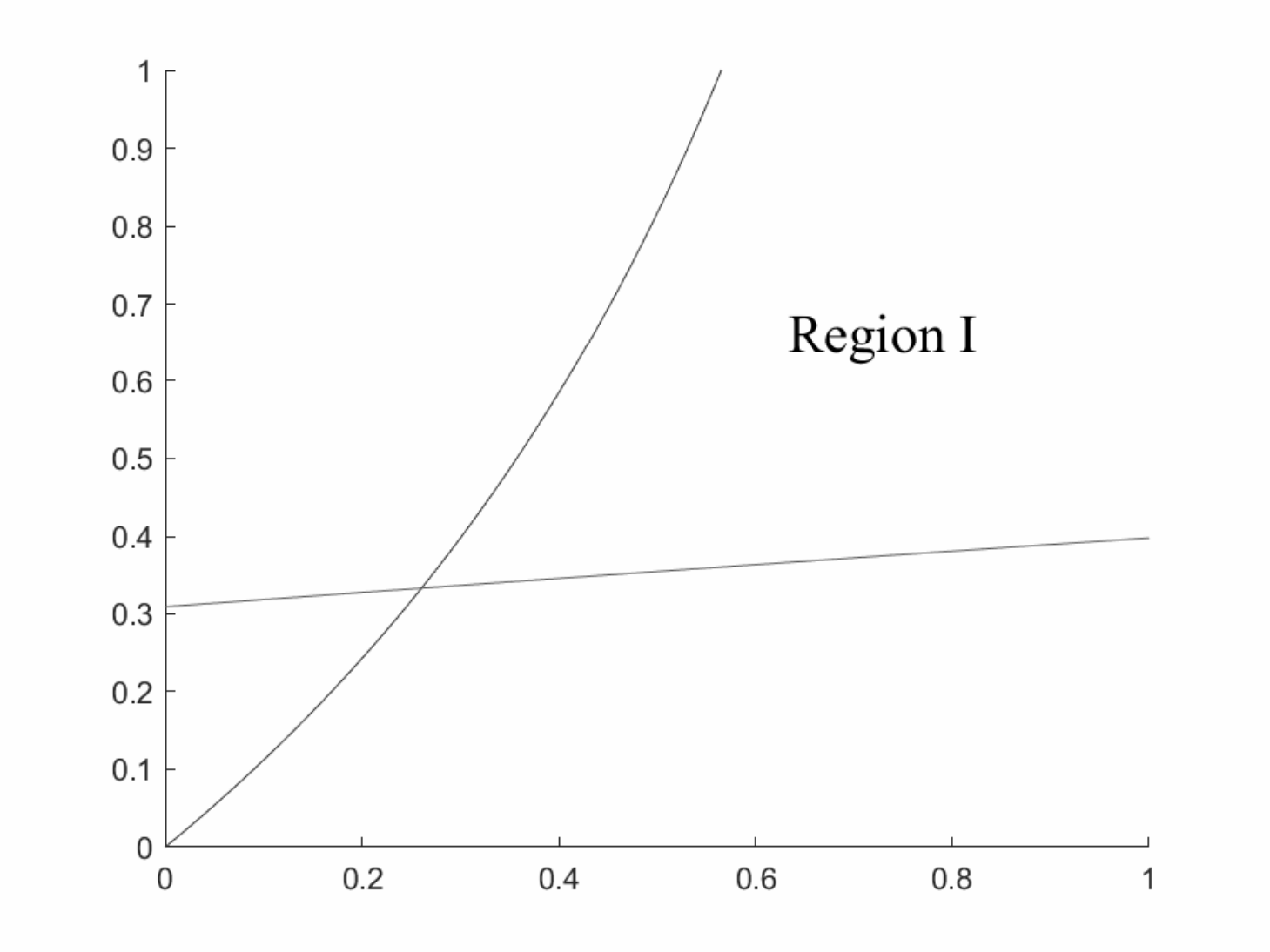}
  \begin{center}
  $z \ = \ 0.25$ \label{fig:trivial25}
  \end{center}
\endminipage\hfill
\minipage{0.32\textwidth}%
  \includegraphics[width=\linewidth]{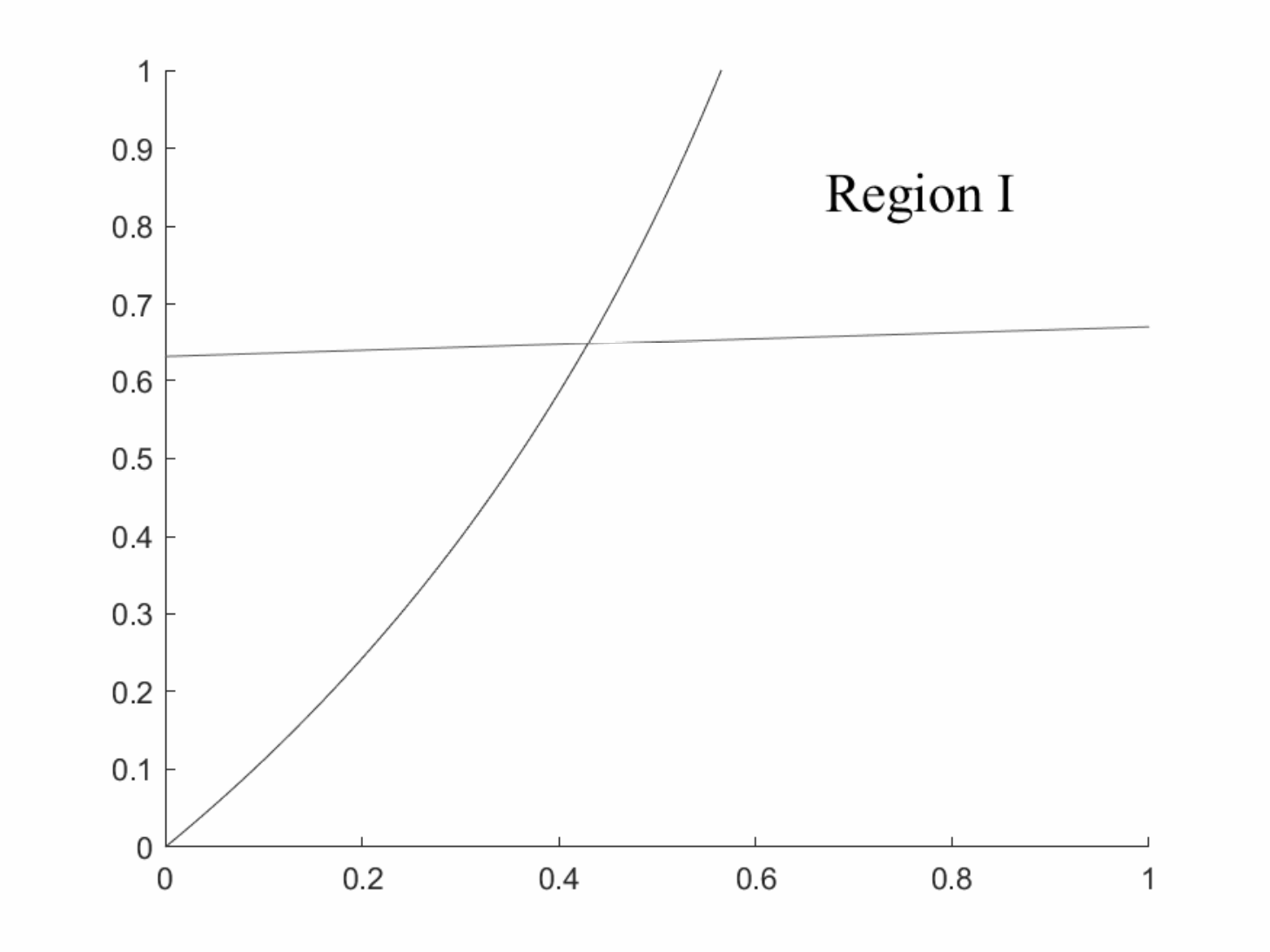}
  \begin{center}
  $z \ = \ 0.75$ \label{fig:trivial75}
  \end{center}
\endminipage
\caption{When $b \le (1-a)/\sqrt{n_1+n_2}$, slices of Region I on the xy-plane at various values of $z$. \\
($b  \ = \  0.08, n_1  \ = \  6, n_2  \ = \  10, a  \ = \  0.5$).} \label{fig:small_b_region_i}
\normalsize
\end{figure}
\FloatBarrier
We break the analysis of $F$ into regions induced by $\phi_1$, $\phi_2$ and $\phi_3$. We focus on the effect of $F$ in Region I (see Figure \ref{fig:small_b_region_i}). The first lemma provides information on the image of this region under $F$, and in the next lemma we use this to show that Region I maps to itself under $F$.
\begin{lemma} \label{decreasing_lemma}
Points $(x,y,z)$ in Region I strictly decrease in $x$, $y$ and $z$ on iteration by $F$.
\end{lemma}
\begin{proof}
A point $(x,y,z)$ in Region I region satisfies the inequalities:
\begin{align}
    x &> \frac{1-(1-by)^{n_1}}{1-a(1-by)^{n_1}} \label{region_i_x} \\
    y &> \frac{1-(1-bx)(1-bz)^{n_2}}{1-a(1-bx)(1-bz)^{n_2}} \label{region_i_y} \\
    z &> \frac{by}{1-a+aby}. \label{region_i_z}
\end{align}
As $a,b \in (0,1)$, $x,y \in (0,1]$ and $n_1 \in \mathbb{Z}^+$, we have
\begin{equation}
    x-ax(1-by)^{n_1} > 1-(1-by)^{n_1},
\end{equation}
hence
\begin{align}
    x &> 1-(1-by)^{n_1} + ax(1-by)^{n_1} \nonumber \\
    & \ = \  1-(1-ax)(1-by)^{n_1} \nonumber \\
    & \ = \  f_1(x,y,z) \label{x_iter_ineq}
\end{align}
Similarly, as $a,b \in (0,1)$, $x,y,z \in (0,1]$ and $n_2 \in \mathbb{Z}^+$, we have
\begin{equation}
    y-ay(1-bx)(1-bz)^{n_2} > 1-(1-bx)(1-bz)^{n_2},
\end{equation}
and thus
\begin{align}
    y &> 1-(1-bx)(1-bz)^{n_2}+ay(1-bx)(1-bz)^{n_2} \nonumber \\
    & \ = \  1-(1-ay)(1-bx)(1-bz)^{n_2} \nonumber \\
    & \ = \  1-(1-ay)(1-bx)(1-bz)^{n_2} \nonumber \\
    & \ = \  f_2(x,y,z).  \label{y_iter_ineq}
\end{align}
Finally, as $a,b \in (0,1)$ and $y,z \in (0,1]$ we have
\begin{equation}
z-az+abyz > by,
\end{equation}
which yields
\begin{align}
    z &> az+by-abyz \nonumber \\
    & \ = \  1-(1-az-by+abyz) \nonumber \\
    & \ = \  1-(1-az)(1-by) \nonumber \\
    & \ = \  f_3(x,y,z). \label{z_iter_ineq}
\end{align}
Thus, the iterate of any point $(x,y,z)$ in Region I by $F$ is strictly decreasing in $x$, $y$ and $z$.
\end{proof}
\begin{lemma}
The image of Region I under $F$ is contained in Region I.
\end{lemma} \label{contained_lemma_1}
\begin{proof}
We prove that for a point $(x,y,z)$ in Region I, its $x$-coodinate iterate satisfies (\ref{region_i_x}), its $y$-coordinate iterate satisfies (\ref{region_i_y}), and its $z$-coordinate iterate satisfies (\ref{region_i_z}).
\paragraph{\newline $x$-coordinate Iteration: \newline}
The $x$-coordinate iterate must satisfy
\begin{equation}
    1-(1-ax)(1-by)^{n_1} \ > \ \frac{1-(1-b(1-(1-ay)(1-bx)(1-bz)^{n_2}))^{n_1}}{1-a(1-b(1-(1-ay)(1-bx)(1-bz)^{n_2}))^{n_1}}.
\end{equation}

As $(x,y,z)$ is in the given region, we have
\begin{equation}
    x > 1-(1-ax)(1-by)^{n_1}.
\end{equation}
As $a,b \in (0,1)$, $x,y \in (0,1]$ and $n_1 \in \mathbb{Z}^+$, it follows that
\begin{equation}
    \frac{x}{1-(1-ax)(1-by)^{n_1}} > 1
\end{equation}
so
\begin{equation}
    1-\frac{x}{1-(1-ax)(1-by)^{n_1}} < 1-1.
\end{equation}
As $a(1-by)^{n_1}>0$,
\begin{equation}
    1-\frac{ax(1-by)^{n_1}}{1-(1-ax)(1-by)^{n_1}} < 1-a(1-by)^{n_1}.
\end{equation}
Simplifying the left hand side of this inequality, we obtain
\begin{align}
    1-\frac{ax(1-by)^{n_1}}{1-(1-ax)(1-by)^{n_1}} &< 1-a(1-by)^{n_1} \nonumber \\
    \frac{1-(1-ax)(1-by)^{n_1}}{1-(1-ax)(1-by)^{n_1}}-\frac{ax(1-by)^{n_1}}{1-(1-ax)(1-by)^{n_1}} &< 1-a(1-by)^{n_1} \nonumber \\
    \frac{1+(ax-1)(1-by)^{n_1}-ax(1-by)^{n_1}}{1-(1-ax)(1-by)^{n_1}} &< 1-a(1-by)^{n_1} \nonumber \\
    \frac{1-(1-by)^{n_1}}{1-(1-ax)(1-by)^{n_1}} &< 1-a(1-by)^{n_1}.
\end{align}
Hence
\begin{align}
    1-(1-by)^{n_1} &< 1-a(1-by)^{n_1}(1-(1-ax)(1-by)^{n_1}) \nonumber \\
    1-(1-ax)(1-by)^{n_1} &< \frac{1-(1-by)^{n_1}}{1-a(1-by)^{n_1}}. \label{inter_ineq_x_1}
\end{align}
Next, as $y > 1-(1-ay)(1-bx)(1-bz)^{n_2}$, $a,b \in (0,1)$, $x,y,z \in [0,1)$ and $n_1,n_2 \in \mathbb{Z}^+$,
\begin{equation}
    (1-by)^{n_1} < (1-b(1-(1-ay)(1-bx)(1-bz)^{n_2}))^{n_1}.
\end{equation}
Let $1-(1-ay)(1-bx)(1-bz)^{n_2}  \ = \  c$. Then, there exists $\delta > 0$ with $0 < c-\delta < c$ such that $c-\delta  \ = \  (1-by)^{n_1}$, and we have
\begin{align}
    \delta - ac\delta &\ > \ a\delta - ac\delta \nonumber \\
    1-c+\delta-ac+ac^2-ac\delta &\ > \ 1-c+a\delta-ac-ac\delta+ac^2 \nonumber \\
    (1-ac)(1-c+\delta) &\ > \ (1-c)(1-a+a\delta) \nonumber \\
    1-c+\delta &\ > \ \frac{(1-c)(1-ac+a\delta)}{1-ac} \nonumber \\
    \frac{1-(c-\delta)}{1-a(c-\delta)} &\ > \ \frac{1-c}{1-ac} \nonumber \\
    \frac{1-(1-by)^{n_1}}{1-a(1-by)^{n_1}} &\ > \ \frac{1-(1-b(1-(1-ay)(1-bx)(1-bz)^{n_2}))^{n_1}}{1-a(1-b(1-(1-ax)(1-by)(1-bz)^{n_2}))^{n_1}}. \label{inter_ineq_x_2}
\end{align}
Thus, from (\ref{inter_ineq_x_1}) and (\ref{inter_ineq_x_2}), we have
\begin{equation}
    1-(1-ax)(1-by)^{n_1} \ > \ \frac{1-(1-b(1-(1-ay)(1-bx)(1-bz)^{n_2}))^{n_1}}{1-a(1-b(1-(1-ay)(1-bx)(1-bz)^{n_2}))^{n_1}}
\end{equation}
as desired.
\paragraph{\newline$y$-coordinate Iteration:\newline}
The $y$-coordinate iterate must satisfy
\begin{multline}
    1-(1-ay)(1-bx)(1-bz)^{n_2} \ > \ \\
    \frac{1-(1-b(1-(1-ax)(1-by)^{n_1})(1-b(1-(1-az)(1-by))^{n_2}}{1-a(1-b(1-(1-ax)(1-by)^{n_1})(1-b(1-(1-az)(1-by))^{n_2}}.
\end{multline}
As  $z > 1-(1-az)(1-by)$, we have
\begin{equation}
1-bz \ < \ 1-b(1-(1-az)(1-by)).
\end{equation}
As $n_2 \in \mathbb{Z}^+$, we obtain
\begin{equation}
(1-bz)^{n_2} \ < \ (1-(1-(1-az)(1-by)))^{n_2}.
\end{equation}
Then, as $1-ay > 0$ and $1-bx > 0$,
\begin{equation}
(1-ay)(1-bx)(1-bz)^{n_2} \ < \ (1-ay)(1-bx)(1-b(1-(1-az)(1-by)))^{n_2},
\end{equation}
so
\begin{multline}
1-(1-ay)(1-bx)(1-bz)^{n_2} \\
\ > \ 1-(1-ay)(1-bx)(1-b(1-(1-az)(1-by)))^{n_2}. \label{inter_ineq_y_3}
\end{multline}
\newline
Furthermore, we have $y > 1-(1-ay)(1-bx)(1-bz)^{n_2}$.
Then,
\begin{align}
y &\ > \ 1-(1-ay)(1-bx)(1-bz)^{n_2} \nonumber \\
&\ > \ 1-(1-ay)(1-bx)(1-b(1-(1-az)(1-by)))^{n_2}
\end{align}
so
\begin{equation}
y \ > \ 1-(1-ay)(1-bx)(1-b(1-(1-az)(1-by)))^{n_2}.
\end{equation}
Hence, we get
\begin{equation}
    \frac{y}{1-(1-ay)(1-bx)(1-b(1-(1-az)(1-by)))^{n_2}} \ > \ 1.
\end{equation}
As
\begin{equation*}
    a(1-bx)(1-b(1-(1-az)(1-by)))^{n_2} \ > \ 0,
\end{equation*}
multiplying this to both sides of the inequality, we get
\begin{multline}
    \frac{ay(1-bx)(1-b(1-(1-az)(1-by)))^{n_2}}{1-(1-ay)(1-bx)(1-b(1-(1-az)(1-by)))^{n_2}} \\
    \ > \ a(1-bx)(1-b(1-(1-az)(1-by)))^{n_2}
\end{multline}
so
\begin{multline}
    1-\frac{ay(1-bx)(1-b(1-(1-az)(1-by)))^{n_2}}{1-(1-ay)(1-bx)(1-b(1-(1-ay)(1-by)))^{n_2}} \\
    \ < \ 1-a(1-bx)(1-b(1-(1-az)(1-by)))^{n_2}.
\end{multline}
Simplifying the left-hand side of this inequality, we get
\begin{multline*}
\frac{1-(1-ay)(1-bx)(1-b(1-(1-az)(1-by)))^{n_2}}{1-(1-ay)(1-bx)(1-b(1-(1-az)(1-by)))^{n_2}} \\
-\frac{ay(1-bx)(1-b(1-(1-az)(1-by)))^{n_2}}{1-(1-ay)(1-bx)(1-b(1-(1-az)(1-by)))^{n_2}} \\
\ < \ 1-a(1-bx)(1-b(1-(1-az)(1-by)))^{n_2}
\end{multline*}
\begin{multline*}
\frac{1+(ay-1)(1-bx)(1-b(1-(1-az)(1-by)))^{n_2}-ay(1-bx)(1-b(1-(1-az)(1-by)))^{n_2}}{1-(1-ay)(1-bx)(1-b(1-(1-az)(1-by)))^{n_2}} \\
\ < \ 1-a(1-bx)(1-b(1-(1-az)(1-by)))^{n_2}
\end{multline*}
\begin{multline}
\frac{1-(1-bx)(1-b(1-(1-az)(1-by)))^{n_2}}{1-(1-ay)(1-bx)(1-b(1-(1-az)(1-by)))^{n_2}} \\
\ < \ 1-a(1-bx)(1-b(1-(1-az)(1-by)))^{n_2}.
\end{multline}
Hence,
\begin{multline}
    1-(1-ay)(1-bx)(1-b(1-(1-az)(1-by)))^{n_2} \\
    \ > \ \frac{1-(1-bx)(1-b(1-(1-az)(1-by)))^{n_2}}{1-a(1-bx)(1-b(1-(1-az)(1-by)))^{n_2}}. \label{inter_ineq_y_1}
\end{multline}
Next, as $x > 1-(1-ax)(1-by)^{n_1}$, there exists $c > 0$ such that $0 < x-c < x$ and $x-c \ = \ 1-(1-ax)(1-by)^{n_1}$.
Then, we have
\begin{align}
    -bc &< -abc \nonumber \\
    -bc(1-b(1-(1-az)(1-by)))^{n_2} &< -abc(1-b(1-(1-az)(1-by)))^{n_2}.
\end{align}
As $a,b \in (0,1)$ and $x,y,z \in (0,1]$,
\begin{align*}
&1-a(1-b(1-(1-az)(1-by)))^{n_2}+abx(1-b(1-(1-az)(1-by)))^{n_2} \\
&-(1-b(1-(1-az)(1-by)))^{n_2}+a(1-b(1-(1-az)(1-by)))^{2n_2} \\
&-abx(1-b(1-(1-az)(1-by)))^{2n_2}+bx(1-b(1-(1-az)(1-by)))^{n_2} \\
&-abx(1-b(1-(1-az)(1-by)))^{2n_2}+ab^2x^2(1-b(1-(1-az)(1-by)))^{2n_2} \\
&+abc(1-b(1-(1-az)(1-by)))^{2n_2}-ab^2cx(1-b(1-(1-az)(1-by)))^{2n_2} \\
&-bc(1-b(1-(1-az)(1-by)))^{n_2}
\end{align*}
\begin{align*}
&\ < \ 1-a(1-b(1-(1-az)(1-by)))^{n_2}+abx(1-b(1-(1-az)(1-by)))^{n_2} \\
&-(1-b(1-(1-az)(1-by)))^{n_2}+a(1-b(1-(1-az)(1-by)))^{2n_2} \\
&-abx(1-b(1-(1-az)(1-by)))^{2n_2}+bx(1-b(1-(1-az)(1-by)))^{n_2} \\
&-abx(1-b(1-(1-az)(1-by)))^{2n_2}+ab^2x^2(1-b(1-(1-az)(1-by)))^{2n_2} \\
&+abc(1-b(1-(1-az)(1-by)))^{2n_2}-ab^2cx(1-b(1-(1-az)(1-by)))^{2n_2} \\
&-abc(1-b(1-(1-az)(1-by)))^{n_2}
\end{align*}
\begin{multline*}
    (1-(1-b(x-c))(1-b(1-(1-az)(1-by)))^{n_2}) \\
    (1-a(1-bx)(1-b(1-(1-az)(1-by)))^{n_2}) \\
    \ < \ (1-a(1-b(x-c))(1-b(1-(1-az)(1-by)))^{n_2}) \\
    (1-(1-bx)(1-b(1-(1-az)(1-by)))^{n_2})
\end{multline*}
\begin{multline}
\frac{1-(1-b(x-c))(1-b(1-(1-az)(1-by)))^{n_2}}{1-a(1-b(x-c))(1-b(1-(1-az)(1-by)))^{n_2}} \\
\ < \ \frac{1-(1-bx)(1-b(1-(1-az)(1-by)))^{n_2}}{1-a(1-bx)(1-b(1-(1-az)(1-by)))^{n_2}}.
\end{multline}
As $x-c \ = \ 1-(1-ax)(1-by)^{n_1}$, substituting this into the inequality, we get
\begin{multline}
\frac{1-(1-b(1-(1-ax)(1-by)^{n_1}))(1-b(1-(1-az)(1-by)))^{n_2}}{1-a(1-b(1-(1-ax)(1-by)^{n_1}))(1-b(1-(1-az)(1-by)))^{n_2}} \\
< \frac{1-(1-bx)(1-b(1-(1-az)(1-by)))^{n_2}}{1-a(1-bx)(1-b(1-(1-az)(1-by)))^{n_2}}. \label{inter_ineq_y_2}
\end{multline}
Hence, from (\ref{inter_ineq_y_1}) and (\ref{inter_ineq_y_2}), we get
\begin{multline}
    1-(1-ay)(1-bx)(1-b(1-(1-az)(1-by)))^{n_2} \\
    > \frac{1-(1-b(1-(1-ax)(1-by)^{n_1}))(1-b(1-(1-az)(1-by)))^{n_2}}{1-a(1-b(1-(1-ax)(1-by)^{n_1}))(1-b(1-(1-az)(1-by)))^{n_2}}. \label{inter_ineq_y_4}
\end{multline}
Furthermore, recall that from (\ref{inter_ineq_y_3}) we have
\begin{equation}
1-(1-ay)(1-bx)(1-bz)^{n_2} > 1-(1-ay)(1-bx)(1-b(1-(1-az)(1-by)))^{n_2}, \nonumber
\end{equation}
so from (\ref{inter_ineq_y_3}) and (\ref{inter_ineq_y_4}) we get
\begin{multline}
1-(1-ay)(1-bx)(1-bz)^{n_2} \\
\ > \ \frac{1-(1-b(1-(1-ax)(1-by)^{n_1}))(1-b(1-(1-az)(1-by)))^{n_2}}{1-a(1-b(1-(1-ax)(1-by)^{n_1}))(1-b(1-(1-az)(1-by)))^{n_2}}
\end{multline}
as desired.
\paragraph{\newline$z$-coordinate Iteration:\newline}
The $z$-coordinate iterate must satisfy
\begin{equation}
    1-(1-az)(1-by) \ > \ \frac{b(1-(1-ay)(1-bx)(1-bz)^{n_2})}{1-a+ab(1-(1-ay)(1-bx)(1-bz)^{n_2})}.
\end{equation}
We have $z \ > \ 1-(1-az)(1-by)$. Then
\begin{equation}
    \frac{z}{1-(1-az)(1-by)} \ > \ 1.
\end{equation}
As $aby-a < 0$,
\begin{equation}
    \frac{z(aby-a)}{1-(1-az)(1-by)} \ < \ -a+aby,
\end{equation}
so
\begin{equation}
    1+\frac{z(aby-a)}{1-(1-az)(1-by)} \ < \ 1-a+aby.
\end{equation}
Simplifying the left-hand side of this inequality, we get
\begin{align}
    \frac{1-(1-az)(1-by)}{1-(1-az)(1-by)}+\frac{z(aby-a)}{1-(1-az)(1-by)} &\ < \ 1-a+aby \nonumber \\
    \frac{az+by-abyz}{1-(1-az)(1-by)}+\frac{abyz-az}{1-(1-az)(1-by)} &\ < \ 1-a+aby \nonumber \\
    \frac{by}{1-(1-az)(1-by)} &\ < \ 1-a+aby.
\end{align}
Hence
\begin{equation}
   \frac{by}{1-a+aby} \ \leq \ 1-(1-az)(1-by). \label{inter_ineq_z_1}
\end{equation}
Next, as $y > 1-(1-ay)(1-bx)(1-bz)^{n_2}$, there exists $c > 0$ such that $0 \ < \ y-c \ < \ y$ and $y-c \ = \ 1-(1-ay)(1-bx)(1-bz)^{n_2}$.
Then, as $a,b \in (0,1)$ and $x,y \in (0,1]$, we get
\begin{align}
    abc-bc &\ < \ 0 \nonumber \\
    by-aby-ab^2y^2-ab^2cy+abc-bc &\ < \ by-aby-ab^2y^2-ab^2cy \nonumber \\
    b(y-c)(1-a+aby) &\ < \ by(1-a+ab(y-c)) \nonumber \\
    \frac{b(y-c)}{1-a+ab(y-c)} &\ < \ \frac{by}{1-a+aby} \nonumber \\
    \frac{b(1-(1-ay)(1-bx)(1-bz)^{n_2})}{1-a+ab(1-(1-ay)(1-bx)(1-bz)^{n_2})} &\ < \ \frac{by}{1-a+aby}. \label{inter_ineq_z_2}
\end{align}
Thus, from (\ref{inter_ineq_z_1}) and (\ref{inter_ineq_z_2}), we have
\begin{equation}
    1-(1-az)(1-by) \ > \ \frac{b(1-(1-ay)(1-bx)(1-bz)^{n_2})}{1-a+ab(1-(1-ay)(1-bx)(1-bz)^{n_2})}
\end{equation}
as desired.
\end{proof}

\subsubsection{Limiting Behavior}
Armed with these two lemmas, we can now examine the limiting behavior of points $(x,y,z)$ under $F$. We first look at the special case where a point is in Region I. We then extend this to make a general statement on the limiting behavior of all points under $F$, when $b < (1-a)/\sqrt{n_1+n_2}$.
\begin{lemma}
When $b \leq (1-a)/\sqrt{n_1+n_2}$, any point in Region I iterates to the trivial fixed point under $F$.
\end{lemma}
\begin{proof}
Consider any non-trivial point $(x_0,y_0,z_0)$ in the given region. Define the sequences $x_{t+1} \ = \ f_1(x_t,y_t,z_t)$, $y_{t+1} \ = \ f_2(x_t,y_t,z_t)$, and $z_{t+1} \ = \ f_3(x_t,y_t,z_t)$. Then, by Lemma \ref{decreasing_lemma}, $\{x\}_{t \ = \ 0}^{\infty}$, $\{y\}_{t \ = \ 0}^{\infty}$, and $\{z\}_{t \ = \ 0}^{\infty}$ are non-increasing sequences. Furthermore, all three sequences are bounded below by $0$. Then, by the Monotone Convergence Theorem, these sequences must converge to $0$. Thus, $(x_0,y_0,z_0)$ iterates to $(0,0,0)$, which is the trivial fixed point.
\end{proof}
We now use this to complete our proof of Theorem 1.1, I(b). We extend the essential method which was used in [BGKMRS] to prove the limiting behavior of points in a 2-level system. Consider any cuboid in $[0,1]^3$ such that one vertex is $(0,0,0)$. Assume the vertex that is the furthest away from this one (the vertex that is across the internal diagonal of the cuboid from $(0,0,0)$) is in Region I. We show that the image of this cuboid under F is strictly contained in the cuboid by showing that the image of the point in Region I has both coordinates smaller than any other iterate. As this vertex iterates to the trivial fixed point since it is in Region I, so too do all the other points in the cuboid, as the lengths of the internal diagonals of the iterations of the cuboid tend to zero.
\begin{proof}[Proof of Theorem 1.1, I(b)]
Consider the cuboid of all points $(x,y,z) \in [0,1]^3$ with $0 \leq x \leq x_u$, $0 \leq y \leq y_u$, and $0 \leq z \leq z_u$ such that $(x_u,y_u,z_u)$ is in Region I. Note that this cuboid is able to encompass all points in $[0,1]^3$, as we can always choose $(x_u,y_u,z_u)  \ = \  (1,1,1)$. Let $p_{0,1}(x,y,z) \ = \ x$, $p_{0,2}(x,y,z) \ = \ y$, and $p_{0,3}(x,y,z) \ = \ z$. Define the sequence
\begin{equation}
    p_t(x,y,z) \ = \ (p_{t,1}(x,y,z),p_{t,2}(x,y,z),p_{t,3}(x,y,z))
\end{equation}
by
\begin{align}
    p_{t+1,1} \ = \ f_1(p_{t,1}(x,y,z),p_{t,2}(x,y,z),p_{t,3}(x,y,z))  \nonumber \\
    p_{t+1,2} \ = \ f_2(p_{t,1}(x,y,z),p_{t,2}(x,y,z),p_{t,3}(x,y,z)) \nonumber \\
    p_{t+1,3} \ = \ f_3(p_{t,1}(x,y,z),p_{t,2}(x,y,z),p_{t,3}(x,y,z)).
\end{align}
We prove by induction that
\begin{align}
    p_{t,1}(0,0,0) \ \leq \ p_{t,1}(x,y,z) \ \leq \ p_{t,1}(x_u,y_u,z_u) \nonumber \\
    p_{t,2}(0,0,0) \ \leq \ p_{t,2}(x,y,z) \ \leq \ p_{t,2}(x_u,y_u,z_u) \nonumber \\
    p_{t,3}(0,0,0) \ \leq \ p_{t,3}(x,y,z) \ \leq \ p_{t,3}(x_u,y_u,z_u).
\end{align}
As our base case follows from any choice of $(x_u,y_u,z_u)$, we proceed to the inductive step. Suppose that $p_{t,1}(x_u,y_u,z_u) \ \geq \ p_{t,1}(x,y,z)$ and $p_{t,2}(x_u,y_u,z_u) \ \geq \ p_{t,2}(x,y,z)$ and $p_{t,3}(x_u,y_u,z_u) \ \geq \ p_{t,3}(x,y,z)$.
\paragraph{\newline$x$-coordinate:\newline}
As $a,b \in (0,1)$, we have
\begin{equation}
    1-ap_{t,1}(x_u,y_u,z_u) \ \leq \ 1-ap_{t,1}(x,y,z)
\end{equation}
and
\begin{equation}
    1-bp_{t,2}(x_u,y_u,z_u) \ \leq \ 1-bp_{t,2}(x,y,z).
\end{equation}
If follows that
\begin{equation}
    (1-ap_{t,1}(x_u,y_u,z_u))(1-bp_{t,2}(x_u,y_u,z_u))^{n_1} \ \leq \ (1-ap_{t,1}(x,y,z))(1-bp_{t,2}(x,y,z))^{n_1}
\end{equation}
so
\begin{multline}
    1-(1-ap_{t,1}(x_u,y_u,z_u))(1-bp_{t,2}(x_u,y_u,z_u))^{n_1} \\
    \ \geq \ 1-(1-ap_{t,1}(x,y,z))(1-bp_{t,2}(x,y,z))^{n_1}.
\end{multline}
Furthermore, as the trivial fixed point stays the same upon iteration by $F$, we have
\begin{equation}
    1-ap_{t,1}(0,0,0)  \ = \  1 \ \geq \ ap_{t,1}(x,y,z)
\end{equation}
and
\begin{equation}
    1-bp_{t,2}(0,0,0)  \ = \  1 \ \geq \ 1-bp_{t,2}(x,y,z).
\end{equation}
Hence
\begin{equation}
    1-(1-ap_{t,1}(0,0,0))(1-bp_{t,2}(0,0,0))^{n_1} \ \leq \ 1-(1-ap_{t,1}(x,y,z))(1-bp_{t,2}(x,y,z))^{n_1},
\end{equation}
and therefore
\begin{equation}
    p_{t+1,1}(0,0,0) \ \leq \ p_{t+1,1}(x,y,z) \ \leq \ p_{t+1,1}(x_u,y_u,z_u).
\end{equation}
\paragraph{\newline$y$-coordinate:\newline}
As $a,b \in (0,1)$, we have
\begin{equation}
    1-ap_{t,2}(x_u,y_u,z_u) \ \leq \ 1-ap_{t,2}(x,y,z)
\end{equation}
and
\begin{equation}
    1-bp_{t,1}(x_u,y_u,z_u) \ \leq \ 1-bp_{t,1}(x,y,z)
\end{equation}
and
\begin{equation}
    1-bp_{t,3}(x_u,y_u,z_u) \ \leq \ 1-bp_{t,3}(x,y,z).
\end{equation}
It follows that
\begin{multline}
    (1-ap_{t,2}(x_u,y_u,z_u))(1-bp_{t,1}(x_u,y_u,z_u))(1-bp_{t,3}(x_u,y_u,z_u))^{n_2} \\
    \ \leq \ (1-ap_{t,2}(x,y,z))(1-bp_{t,1}(x,y,z))(1-bp_{t,3}(x,y,z))^{n_2},
\end{multline}
so
\begin{multline}
    1-(1-ap_{t,2}(x_u,y_u,z_u))(1-bp_{t,1}(x_u,y_u,z_u))(1-bp_{t,3}(x_u,y_u,z_u))^{n_2} \\
    \ \geq \ 1-(1-ap_{t,2}(x,y,z))(1-bp_{t,1}(x,y,z))(1-bp_{t,3}(x,y,z))^{n_2}.
    \end{multline}
Furthermore, as the trivial fixed point stays the same upon iteration by $F$, we have
\begin{equation}
    1-ap_{t,2}(0,0,0)  \ = \  1 \ \geq \ ap_{t,2}(x,y,z)
\end{equation}
and
\begin{equation}
    1-bp_{t,1}(0,0,0)  \ = \  1 \ \geq \ 1-bp_{t,1}(x,y,z)
\end{equation}
and
\begin{equation}
    1-bp_{t,3}(0,0,0)  \ = \  1 \ \geq \ 1-bp_{t,3}(x,y,z).
\end{equation}
Hence
\begin{multline}
    1-(1-ap_{t,2}(0,0,0))(1-bp_{t,1}(0,0,0))(1-bp_{t,3}(0,0,0))^{n_2} \\
    \ \leq \ 1-(1-ap_{t,2}(x,y,z))(1-bp_{t,1}(x,y,z))(1-bp_{t,3}(x,y,z))^{n_2},
\end{multline}
and therefore
\begin{equation}
    p_{t+1,2}(0,0,0) \ \leq \ p_{t+1,2}(x,y,z) \ \leq \ p_{t+1,2}(x_u,y_u,z_u).
\end{equation}
\paragraph{\newline$z$-coordinate:\newline}
As $(a,b) \in (0,1)$, we have
\begin{equation}
    1-ap_{t,3}(x_u,y_u,z_u) \ \leq \ 1-ap_{t,3}(x,y,z)
\end{equation}
and
\begin{equation}
    1-bp_{t,2}(x_u,y_u,z_u) \ \leq \ 1-bp_{t,2}(x,y,z).
\end{equation}
It follows that
\begin{equation}
    (1-ap_{t,3}(x_u,y_u,z_u))(1-bp_{t,2}(x_u,y_u,z_u)) \ \leq \ (1-ap_{t,3}(x,y,z))(1-bp_{t,2}(x,y,z))
\end{equation}
so
\begin{equation}
    1-(1-ap_{t,3}(x_u,y_u,z_u))(1-bp_{t,2}(x_u,y_u,z_u)) \ \geq \ 1-(1-ap_{t,3}(x,y,z))(1-bp_{t,2}(x,y,z)).
\end{equation}
Furthermore, as the trivial fixed point stays the same upon iteration by $F$, we have
\begin{equation}
    1-ap_{t,3}(0,0,0) \ \geq \ 1-ap_{t,3}(x,y,z)
\end{equation}
and
\begin{equation}
    1-bp_{t,2}(0,0,0) \ \geq \ 1-bp_{t,2}(x,y,z).
\end{equation}
Hence
\begin{equation}
    1-(1-ap_{t,3}(0,0,0))(1-bp_{t,2}(0,0,0)) \ \leq \ 1-(1-ap_{t,3}(x,y,z))(1-bp_{t,2}(x,y,z)).
\end{equation}
Thus
\begin{equation}
    p_{t+1,3}(0,0,0) \ \leq \ p_{t+1,3}(x,y,z) \ \leq \ p_{t+1,3}(x_u,y_u,z_u).
\end{equation}
Therefore, the claimed inequalities hold for all $(x,y,z) \in [0,1]^3$ by induction.
\paragraph{\newline Limiting behavior:\newline}
Now, taking the limit, we get
\begin{equation}
    \lim_{t \to \infty}p_{t,1}(0,0,0) \ \leq \ \lim_{t \to \infty}p_{t,1}(x,y,z) \ \leq \ \lim_{t \to \infty}p_{t,1}(x_u,y_u,z_u),
\end{equation}
\begin{equation}
    \lim_{t \to \infty}p_{t,2}(0,0,0) \ \leq \ \lim_{t \to \infty}p_{t,2}(x,y,z) \ \leq \ \lim_{t \to \infty}p_{t,2}(x_u,y_u,z_u),
\end{equation}
\begin{equation}
    \lim_{t \to \infty}p_{t,3}(0,0,0) \ \leq \ \lim_{t \to \infty}p_{t,3}(x,y,z) \ \leq \ \lim_{t \to \infty}p_{t,3}(x_u,y_u,z_u).
\end{equation}
As $(0,0,0)$ is the trivial fixed point under $F$ and $(x_u,y_u,z_u)$ is in Region I, from Lemma \ref{contained_lemma_1} we obtain
\begin{equation}
    0 \ \leq \ \lim_{t \to \infty}p_{t,1}(x,y,z) \ \leq \ 0,
\end{equation}
\begin{equation}
    0 \ \leq \ \lim_{t \to \infty}p_{t,2}(x,y,z) \ \leq \ 0,
\end{equation}
\begin{equation}
    0 \ \leq \ \lim_{t \to \infty}p_{t,3}(x,y,z) \ \leq \ 0.
\end{equation}
Thus,
\begin{align}
    &\lim_{t \to \infty}p_{t,1}(x,y,z)  \ = \  0, \\
    &\lim_{t \to \infty}p_{t,2}(x,y)  \ = \  0, \\
    &\lim_{t \to \infty}p_{t,3}(x,y,z)  \ = \  0,
\end{align}
that is,
\begin{equation}
    \lim_{t \to \infty} p_t(x,y,z)  \ = \  (0,0,0).
\end{equation}
Therefore, when $b \ \leq \ (1-a)/\sqrt{n_1+n_2}$, all points $(x,y,z) \in [0,1]^3$ iterate to the trivial fixed point under $F$.
\end{proof}

The goal of this section was to highlight the proof method of specifying ``regions'' in $[0,1]^2$ induced by the two partial fixed point curves, and then using the squeeze theorem along with properties of points within the induced regions to show convergence to a fixed point. The proof of Theorem 3.2, II(b) is analogous to the method above to prove Theorem 3.2, I(b).
\end{subequations}

\begin{subequations}
\section{Extension to $k$-level Starlike Graphs}
\subsection{Shape of $F$ for $k$-level Starlike Graphs}
\hfill\\

So far, we have explored the characteristics and limiting behavior of this system when applied to $3$-level starlike graphs. We now examine the same system when applied to a starlike graph with an arbitrary number of levels, in other words, a $k$-level starlike graph. We keep the simplifying assumption that at each level, the number of spokes is the same. For the earlier $3$-level case, we considered a graph with $n_1$ $2$-level spoke nodes around a hub node, and $n_2$ $3$-level spoke nodes around each of the $n_1$ $2$-level spoke nodes. In a $k$-level system, we assume there are $n_{k+1}$ $(k+1)$-level spoke nodes connected to each $k$-level spoke node. Similar to the $3$-level case, by Lemma \ref{common_behavior_lemma}, all nodes on the same level approach a common limiting value. Hence, in the $k$-level case, we have a system in $k$ unknowns. Let $d_1$ be the probability that the central hub node is infected, and let $d_2, \dots, d_k$ be the probabilities that $2, \dots, k$-level spoke nodes are infected. Then, we get the following system:
$$F\begin{pmatrix}
d_1\\
d_2\\
d_3\\
\vdots \\
d_k
\end{pmatrix} \ = \ \begin{pmatrix}
f_1(d_1,\dots,d_k) \\
f_2(d_1,\dots,d_k)\\
f_3(d_1,\dots,d_k)\\
\vdots\\
f_k(d_1,\dots,d_k)
\end{pmatrix} \ = \ \begin{pmatrix}
1-(1-ad_1)(1-bd_2)^{n_1}\\
1-(1-ad_2)(1-bd_1)(1-bd_3)^{n_2}\\
1-(1-ad_3)(1-bd_2)(1-bd_4)^{n_3}\\
\vdots\\
1-(1-ad_k)(1-bd_{k-1})
\end{pmatrix}.
$$
Note the similarities to the $3$-level system. In both systems, $f_1$ are described by functions of the same form, while $f_2$ in the $3$-level system is similar to $f_2, \dots, f_{k-1}$ in the $k$-level system. $f_3$ in the $3$-level system looks a lot like $f_k$ in the $k$-level system as well. In other words, as we add more levels, the iterates of the first and last variables stay the same, while all iterates of the variables between the two are described by a common function shape. This allows us to examine the $k$-level system in a similar way to that of the $3$-level system, essentially "collapsing" all of the middle iterates into one group to be analyzed simultaneously. The findings from the $3$-level case are extended to the following.
\begin{theorem} \label{k-level_thr}
\begin{enumerate}[I.]
\item
If $b \leq (1-a)/\sqrt{n_1+\dots+n_{k-1}}$, then
\begin{enumerate}[(a)]
\item
the unique fixed point of $F$ is $(0,0,\dots,0)$.
\item
the system converges to this fixed point, in other words, the virus dies out.
\end{enumerate}
\item
If $b > (1-a)/\sqrt{n_1+\dots+n_{k-1}}$, then
\begin{enumerate}[(a)]
\item
$F$ has a unique, non-trivial fixed point $({d_1}_f,{d_2}_f,\dots,{d_k}_f)$.
\item
the system converges to this non-trivial fixed point.
\end{enumerate}
\end{enumerate}
\end{theorem}

\subsection{Additional Convexity Arguments for Extension to $k$-level}
\hfill\\

As in the $3$-level case, we first determine the location and number of fixed points. We look for partial fixed points by solving the equations
\begin{align}
d_1 & \ = \  f_1(d_1,\dots,d_k) \nonumber \\
d_2 & \ = \  f_2(d_1,\dots,d_k) \nonumber \\
\vdots \nonumber \\
d_k & \ = \  f_k(d_1,\dots,d_k).
\end{align}
This results in the $\phi$ functions
\begin{equation}
\phi_1(d_2,\dots,d_k)  \ = \  d_1  \ = \  \frac{1-(1-bd_2)^{n_1}}{1-a(1-bd_2)^{n_1}}.
\end{equation}
For all levels $2 \leq m \leq k-1$,
\begin{equation}
\phi_m(d_1,\dots,d_{m-1},d_{m+1},\dots,d_k)  \ = \  d_m  \ = \  \frac{1-(1-bd_{m-1})(1-bd_{m+1})^{n_m}}{1-a(1-bd_{m-1})(1-bd_{m+1})^{n_m}},
\end{equation}
and for the $k$th level
\begin{equation}
    \phi_k(d_1,\dots,d_{k-1})  \ = \  d_k  \ = \  \frac{bd_{k-1}}{1-a+abd_{k-1}}.
\end{equation}
We then take the compositions $d_1  \ = \  \phi_{1,2,\dots,k}  \ = \  \phi_1$ and $d_2  \ = \ \phi_{2,3,\dots,k}  \ = \  \phi_2 \circ \cdots \circ \phi_k$ to reduce the problem to a two-dimensional one. We now would like to determine the concavity of these two curves. However, complications arise due to the composition of several $\phi$'s. Hence, additional arguments are required.
\begin{lemma} \label{one_var_lemma}
When $k \geq 3$, the composition $\phi_{k-m}(d_1,d_3,\dots,d_k) \circ \cdots \circ \phi_k(d_1,\dots,d_{k-1})$ is a function of one variable $d_{k-m-1}$ for all positive integers $1\leq m \leq k-2$.
\end{lemma}
\begin{proof}
We prove this by induction. For our base case $m \ = \ 1$, we have
\begin{equation}
    \phi_k(d_1,\dots,d_{k-1})  \ = \  d_k  \ = \  \frac{bd_{k-1}}{1-a+abd_{k-1}}
\end{equation}
and
\begin{equation}
   \phi_{k-1}(d_1,\dots,d_{l-1},d_{l+1},\dots,d_k)  \ = \  d_{k-1}  \ = \  \frac{1-(1-bd_{k-2})(1-bd_k)^{n_{k-1}}}{1-a(1-bd_{k-2})(1-bd_k)^{n_{k-1}}}.
\end{equation}
Substituting $d_k$ into $ \phi_{k-1}(d_1,\dots,d_{l-1},d_{l+1},\dots,d_k)$, we obtain
\begin{equation}
    \phi_{k-1}(d_1,\dots,d_{l-1},d_{l+1},\dots,d_k)  \ = \  d_{k-1}  \ = \  \frac{1-(1-bd_{k-2})(1-b(\frac{bd_{k-1}}{1-a+abd_{k-1}}))^{n_{k-1}}}{1-a(1-bd_{k-2})(1-b(\frac{bd_{k-1}}{1-a+abd_{k-1}}))^{n_{k-1}}}.
\end{equation}
Note that the equation contains only the variables $d_{k-1}$ and $d_{k-2}$, so the composition of $\phi_{k-1}$ and $\phi_k$ can be written as a function of $d_{k-2}$. Hence, the claim holds for the base case.
Now suppose the claim holds for $m \ = \ j$. For $m \ = \ j+1$, we have
\begin{multline}
    \phi_{k-(j+1)}(d_1,\dots,d_{k-j-2},d_{k-j},\dots,d_k)  \ = \  d_{k-(j+1)}  \ = \  d_{k-j-1} \\
     \ = \  \frac{1-(1-bd_{k-(j+1)-1})(1-bd_{k-(j+1)+1})^{n_{k-(j+1)}}}{1-a(1-bd_{k-(j+1)-1})(1-bd_{k-(j+1)+1})^{n_{k-(j+1)}}} \\
     \ = \  \frac{1-(1-bd_{k-j-2})(1-bd_{k-j})^{n_{k-j-1}}}{1-a(1-bd_{k-j-2})(1-bd_{k-j})^{n_{k-j-1}}}. \label{phi_k-j-1}
\end{multline}
From our inductive hypothesis, as $\phi_{k-j} \circ \cdots \circ \phi_k = d_{k-j}$ can be written as a function of $d_{k-j-1}$, substituting this into (\ref{phi_k-j-1}) in place of $d_{k-j}$ results in an equation with two variables $d_{k-j-1}$ and $d_{k-j-2}$, so the composite function $\phi_{k-(j+1)} \circ \cdots \circ \phi_k  \ = \  d_{k-j-1}$ can be written as a function of $d_{k-j-2} = d_{k-(m+1)-1}$, as desired.
\end{proof}
We now know that $\phi_{2,\dots,k}$ can be written as a function of $d_1$. Then, as in the $3$-level case, we can draw the curves $d_1 = \phi_1(d_2)$ and $d_2 = \phi_{2,\dots,k}(d_1)$ in $\mathbb{R}^2$. The intersections of these two curves are the desired fixed points. Furthermore, as $\phi_2, \dots, \phi_{k-1}$ have the same functional shape as $\phi_2$ in the $3$-level case,we appeal to our results in Lemma \ref{3_level_fixed_points_lemma} to conclude that $\phi_2, \dots, \phi_{k-1}$ are functions that are non-decreasing in each argument (see (\ref{non-decreasing_calc_start}) $\sim$ (\ref{non-decreasing_calc_end})), and are concave. Since $\phi_1$ from the $3$-level case is equivalent to that in the $k$-level case, and $\phi_3$ in the $3$-level case is equivalent to $\phi_k$ in the $k$-level case, we also know that $\phi_1$ is convex, and $\phi_k$ is concave. We can now use these facts along with Lemma \ref{concavity_lemma} to show that the curve $d_2 = \phi_{2,\dots,k}(d_1)$ is concave.
\begin{lemma} \label{comp_concavity_lemma}
Let $k \geq 3 \in \mathbb{R}$. Then, for all integers $m \in [1,k-2]$, the composition $\phi_{k-m} \circ \cdots \circ \phi_k  \ = \  \phi_{2,\dots,k}$ is concave.
\end{lemma}
\begin{proof}
We prove this by induction. We have that $\phi_{k-1}$ is a concave function defined on $[0,1]^2 \to [0,1]$ that is non-decreasing in each argument. We also have that $\phi_k$ is a concave function defined on $[0,1] \to [0,1]$. Then, from Lemma \ref{concavity_lemma}, $\phi_{k-1,k}$ is concave. From Lemma \ref{one_var_lemma}, we also have that this is a function depending on $d_{k-2}$, so it is a function defined on $[0,1] \to [0,1]$.
Now, suppose that $\phi_{k-m,k}$ is a concave function. Then, as $\phi_{k-m-1}$ is a concave function on $[0,1]^2 \to [0,1]$ that is non-decreasing in each argument and $\phi_{k-m,k}$ is a concave function defined on $[0,1] \to [0,1]$, from Lemma \ref{concavity_lemma}, $\phi_{k-m-1,k}$ is concave as well. Therefore, our claim is proven.
\end{proof}
We now know that $\phi_{2,\dots,k}$ is concave, and that it is a function that depends only on $d_1$. This allows the existence and uniqueness of fixed points in the $k$-level system to be proved in the same manner as the $3$-level case. The proofs of convergence to the fixed points are also analogous to the $3$-level case.
\end{subequations}
\hfill\\

\section{Conclusion}
We have presented a model of virus propagation that is applicable to many real-world phenomena, such as the spread of diseases in regions containing one major population center surrounded by numerous dependent areas, and explored extensions and generalizations of it to various classes of starlike graphs. Taking the initial star graph model introduced by \cite{bgkmrs}, we see that with some additional work, it is possible to extend this to $3$-level starlike graphs with the simplifying condition that all nodes have the same number of neighbors on each level. With an additional concavity argument, it was not difficult to extend our model on $3$-level starlike graphs to starlike graphs with an arbitrary number of levels, as long as the simplifying assumption is kept.

An even more realistic model can be imagined as a mesh network of multiple interconnected starlike subgraphs. For example, returning to the example of disease spread, such a model would be able to depict the spread of disease in regions with multiple major population centers, such as the Boston-Washington D.C. megalopolis, and numerous dependent areas in between them. However, such a generalization would require giving up our simplifying assumption, making an analysis challenging. There is still much work to do in this regard.

Another interesting topic of study would be to explore the paths that iterates take when converging to the fixed point. \cite{bgkmrs} conjectured for a $2$-level starlike graph that points that are inside different regions would exhibit different behaviors when iterating to the fixed points. It would be interesting to find simple conditions to support this possibility, and extend the analysis to higher level starlike graphs.

Finally, as the SIS model is a simple model for virus transmission, there are some infectious diseases that may not be adequately modeled using our model. Exploring extensions to our model to account for other properties of infectious diseases, such as changing transmission probabilities with time, multiple infected classes, an exposed class and temporary immunity would make it more useful in applications to epidemiology.


\medskip

\begin{thebibliography}{BGKMRS1}
\bibitem[AM]{am}
\newblock \emph{R.M. Anderson, R.M. May}, Coevolution of hosts and parasites, Parasitology (1982), 411--426.

\bibitem[BGKMRS]{bgkmrs}
\newblock \emph{T. Becker, A. Greaves-Tunnell, A. Kontorovich, S.J. Miller, P. Ravikumar, K. Shen}, Virus Dynamics on Starlike Graphs,  Journal of Nonlinear Systems and Applications (2013), 53--63.

\bibitem[Het]{het}
\newblock \emph{H.W. Hethcote, The Mathematics of Infectious Diseases}, SIAM Rev., 42(4) (2000), 599--653.

\bibitem[KM]{km}
\newblock \emph{W.O. Kermack, A.G. McKendrick, A Contribution to the Mathematical Theory of Epidemics}, Proceedings of the Royal Society of London,  Series A, Containing Papers of a Mathematical and Physical Character (1927), 700--721.

\bibitem[LoAcEl]{loacel}
\newblock \emph{I.M. Longini Jr., E. Ackerman, L.R. Elveback}, An optimization model for influenza A epidemics, Mathematical Biosciences, Volume 38, Issues 1--2 (1978), 141--157.

\bibitem[LY]{ly}
\newblock \emph{W.P. London, J.A. Yorke}, Rucurrent outbreaks of measbles, chickenpox and mumps. I. Seasononal variation in contact rates, American Journal of Epodemiology (1973), 453--468.

\bibitem[McK]{mck}
\newblock \emph{A.G. McKendrick}, Applications of mathematics to medical problems, Proceedings of Edin. Math. Society 14 (1926), 98–-130.

\bibitem[RiFoLa]{rifola}
\newblock \emph{M. Ripeanu, I. Foster, A. Iamnitchi}, Mapping the gnutella network: Properties of large scale peer-to peer systems and implications for system design, IEEE Internet Computing Journal 6 (2001), no. 1, 50-–57.

\bibitem[Rud]{rud}
\newblock \emph{W. Rudin}, Principles of Mathematical Analysis, 3rd edition, International Series in Pure and Applied Mathematics, McGraw-Hill, New York, 1976.

\bibitem[WDWF]{wdwf}
\newblock \emph{Y. Wang, D. Chakrabarti, C. Wang and C. Faloutsos}, Epidemic Spreading in Real Networks: An Eigenvalue Viewpoint, 22nd International Symposium on Reliable Distributed Systems, Proceedings (2003).

\bibitem[YNPM]{ynpm}
\newblock \emph{Y.A. Yorke, N. Nathanson, G. Pianigiani, J. Martin}, Seasonality and the requirements for perpetuation and eradication of viruses in populations, American Journal of Epidemiology (1979), 103--123.

\bibitem[CMA]{cma}
\newblock \emph{I. Cooper, A. Mondal, C.G. Antonopoulos}, A SIR model assumption for the spread of COVID-19 in different communities, Chaos, Solitons and Fractals (2020).

\bibitem[AERRA]{aerra}
\newblock \emph{N. Ahmed, A. Elsonbaty, A. Raza, M. Rafiq, W. Adel}, Numerical simulation and stability analysis of a novel reaction–diffusion COVID-19 model, Nonlinear Dynamics (2021), 106(2):1293--1310.

\bibitem[CLCL]{clcl}
\newblock \emph{Y.C. Chen, P.E. Lu, C.S. Chang, T.H. Liu}, A Time-Dependent SIR Model for COVID-19 With Undetectable Infected Persons, IEEE transactions on network science and engineering (2020), 7(4):3279--3294.

\end{thebibliography}
\end{document}